\documentclass[11pt,a4paper]{amsart}

\usepackage[margin=1in]{geometry}
\usepackage{amssymb,amsfonts,amsthm,amsmath}
\usepackage{enumerate}

\usepackage{xcolor}
\definecolor{ggr}{rgb}{0.05,0.47,0.19}
\usepackage[hyperindex,breaklinks,colorlinks=false,allbordercolors=ggr,pagebackref=true]{hyperref}

\usepackage{amsrefs}

\usepackage[f]{esvect}

\usepackage[utf8]{inputenc}\usepackage[UKenglish]{babel}
\usepackage{graphicx,color,amsthm,xspace,stmaryrd,mathrsfs}\usepackage{microtype}\usepackage{amsmath}\usepackage{xstring}
\usepackage{verbatim} % Allows block-comments.

\newtheorem{Theorem}{Theorem}[section]
\newtheorem{Proposition}[Theorem]{Proposition}
\newtheorem{Lemma}[Theorem]{Lemma}
\newtheorem{Corollary}[Theorem]{Corollary}
\newtheorem*{Claim}{Claim}

\theoremstyle{definition}
\newtheorem{Definition}[Theorem]{Definition}

\newtheorem{Question}[Theorem]{Question}
\newtheorem{Example}[Theorem]{Example}
\newtheorem{Remark}[Theorem]{Remark}

\DeclareMathOperator{\dom}{\mathrm{dom}}
\DeclareMathOperator{\ran}{\mathrm{ran}}
\DeclareMathOperator{\id}{\mathrm{id}}
\DeclareMathOperator{\IN}{\mathrm{IN}}
\DeclareMathOperator{\OUT}{\mathrm{OUT}}

\usepackage{stmaryrd}

\newcommand{\andd}{\wedge}
\newcommand{\la}{\langle}
\newcommand{\ra}{\rangle}
\newcommand{\da}{\!\downarrow}
\newcommand{\ua}{\!\uparrow}
\newcommand{\imp}{\rightarrow}

\newcommand{\biimp}{\leftrightarrow}
\newcommand{\smf}{\smallfrown}

\newcommand{\leqT}{\leq_\mathrm{T}}
\newcommand{\geqT}{\geq_\mathrm{T}}

\newcommand{\leqW}{\leq_\mathrm{W}}
\newcommand{\leW}{<_\mathrm{W}}
\newcommand{\nleqW}{\nleq_\mathrm{W}}
\newcommand{\equivW}{\equiv_\mathrm{W}}

\newcommand{\leqsW}{\leq_\mathrm{sW}}
\newcommand{\lesW}{<_\mathrm{sW}}
\newcommand{\nleqsW}{\nleq_\mathrm{sW}}
\newcommand{\equivsW}{\equiv_\mathrm{sW}}

\newcommand{\MP}[1]{\mathcal{#1}}

\newcommand{\str}{\mathrm{str}}
\newcommand{\alt}{\mathrm{alt}}

\usepackage{xcolor}	
	\usepackage{soul}

	\def\paul#1{\sethlcolor{yellow}\hl{#1}}
	
\def\MLR{\ensuremath{\mathrm{MLR}}\xspace}
\def\rd{\ensuremath{\mathrm{rd}}\xspace}

\def\N{\ensuremath{\mathbb{N}}\xspace}

\def\C{\ensuremath{\mathrm{C}}\xspace}
\def\UC{\ensuremath{\mathrm{UC}}\xspace}

\def\N{\mathbb{N}}

\def\C{\mathrm{C}}

\def\dom{\mathrm{dom}}
\def\ran{\mathrm{ran}}

\def\RD{\ensuremath{\mathsf{RD}}\xspace}
\def\LAY{\ensuremath{\mathsf{LAY}}\xspace}

\let\temp\phi
\let\phi\varphi
\let\varphi\temp

\title{Universality, optimality, and randomness deficiency%
}

\author{Rupert H\" olzl}
\address{Department of Mathematics, Faculty of Science, National University of Singapore, Block S17, 10~Lower Kent Ridge Road, Singapore 119076, Republic of Singapore}
\email{r@hoelzl.fr}
\urladdr{http://hoelzl.fr}

\author{Paul Shafer}
\address{Department of Mathematics,
Ghent University,
Krijgslaan 281, S22,
9000 Ghent,
Belgium}
\email{paul.shafer@ugent.be}
\urladdr{http://cage.ugent.be/~pshafer/}

\thanks{Rupert H\"olzl was supported by a Feodor Lynen postdoctoral research fellowship of the Alexander von Humboldt Foundation and is supported by the Ministry of Education of Singapore through grant R146-000-184-112 (MOE2013-T2-1-062).  Paul Shafer is an FWO Pegasus Long Postdoctoral Fellow.  He was also supported by the Fondation Sciences Math\'ematiques de Paris, and he received travel support from the Bayerisch-Französisches Hochschulzentrum/Centre de Coopération Universitaire Franco-Bavarois.}

\date{\today}
\begin{document}
\begin{abstract}
A Martin-L\"of test $\MP U$ is \emph{universal} if it captures all non-Martin-L\"of random sequences, and it is \emph{optimal} if for every ML-test $\MP V$ there is a $c \in \omega$ such that $\forall n(\MP {V}_{n+c} \subseteq \MP{U}_n)$.  We study the computational differences between universal and optimal ML-tests as well as the effects that these differences have on both the notion of layerwise computability and the Weihrauch degree of $\LAY$, the function that produces a bound for a given Martin-L\"of random sequence's randomness deficiency.  We prove several robustness and idempotence results concerning the Weihrauch degree of $\LAY$, and we show that layerwise computability is more restrictive than Weihrauch reducibility to $\LAY$.
Along similar lines we also study the principle $\RD$, a variant of $\LAY$ outputting the precise randomness deficiency of sequences instead of only an upper bound as $\LAY$.
\end{abstract}

\maketitle

\section{Introduction}

Hoyrup and Rojas~\cite{MR2545900} fix a universal Martin-L\"of test and define a function to be \emph{layerwise computable} if it is computable on Martin-L\"of random inputs when given what essentially amounts to a bound for the input's randomness deficiency as advice.  The ML-test $\MP U = (\MP{U}_n)_{n \in \omega}$ that Hoyrup and Rojas use to define layerwise computability has the special property that for every ML-test $\MP V = (\MP{V}_n)_{n \in \omega}$ there is a $c \in \omega$ such that $\forall n(\MP {V}_{n+c} \subseteq \MP{U}_n)$.  Miyabe~\cite{MR2876979} studies these special, so called \emph{optimal}, tests.  If $\MP U$ and $\MP V$ are two optimal ML-tests, then it is straightforward to see that the notion of layerwise computability is the same when defined via $\MP U$ as it is when defined via $\MP V$.  However, the following example (essentially due to Miyabe, though with a slightly different proof) shows that there are universal ML-tests that are not optimal.

\begin{Example}\label{ex-UniversalNotOptimal}
Let $\MP U$ be any universal ML-test.  Define a test $\MP V$ via $\MP V_n=\bigcap_{i\leq n} \MP U_n$ for all $n \in \omega$, thus making the test a descending chain.  $\MP V$ is also a universal ML-test, so $\lambda(\MP V_n)\not= 0$ for all $n$.  On the other hand $\lim_n \lambda(\MP V_n)=0$.  Therefore there are infinitely many $n$ with $\lambda(\MP V_{n+1})<\lambda(\MP V_n)$.  Assume for the sake of argument that there are infinitely many $n$ with $\lambda(\MP V_{2n+1})<\lambda(\MP V_{2n})$, and let $I$ be the set of these $n$.   (The case in which there are infinitely many $n$ with $\lambda(\MP V_{2n})<\lambda(\MP V_{2n-1})$ is analogous.)

Define an ML-test $\MP W$ via $\MP W_n=\MP V_{2n+1}$ for all $n$.  Clearly $\MP W$ meets the effectivity and measure conditions for being an ML-test.  As $\bigcap_{n \in \omega} \MP{W}_n = \bigcap_{n \in \omega} \MP{V}_n = \bigcap_{n \in \omega} \MP{U}_n$, $\MP W$ is also a universal ML-test.

Fix any $c \in \omega$.  For any $n \in I$ with $n \geq c$, we have that
\begin{align*}
\MP U_{n+c} \supseteq \MP V_{n+c} \supseteq \MP V_{2n} \supsetneq \MP V_{2n+1} = \MP W_n,
\end{align*}
which implies that $\MP U_{n+c} \not \subseteq \MP W_n$. 

That is, for every $c$ there are infinitely many $n$ with $\MP U_{n+c} \not \subseteq \MP W_n$. Thus $\MP W$ is a universal ML-test that is not optimal.\qed
\end{Example}

Miyabe~\cite{MR2876979} obtains a compelling computational difference between optimal ML-tests and universal ML-tests:  by a result of Merkle, Mihailovi\'c, and Slaman~\cite{MR2256638}, there is a universal ML-test $\MP U$ and a left-c.e.\ real $\alpha$ such that $\forall n(\lambda(\MP{U}_n) = 2^{-n}\alpha)$.  Miyabe proves that no optimal ML-test is of this form.

This article presents further differences between optimal ML-tests and universal ML-tests.  If $\MP U$ is an optimal ML-test, then for every ML-test  there trivially is a function $f$ (in fact, a computable function $f$) such that $\forall n(\MP{V}_{f(n)} \subseteq \MP{U}_n)$.  In Section~\ref{sec-NonOptimal}, we show that if $\MP U$ is universal but not optimal, then such an $f$ need not exist; and furthermore that there exist universal ML-tests $\MP U$ and $\MP V$ such that functions $f$ as above do indeed exist, but such that all of these $f$ are difficult to compute.

In Section~\ref{sec-LWC}, we ask if the notion of layerwise computability remains the same if we allow it to be defined using \emph{any}, possibly non-optimal, universal ML-test.  The answer is negative.  It is possible to construct universal ML-tests that distort the randomness deficiencies assigned by a given ML-test quite chaotically.  Likewise, we study the difference between the class of layerwise computable functions and the class of \emph{exactly layerwise computable} functions, where we say that a function on MLR is exactly layerwise computable if it is uniformly computable given an ML-random sequence and its randomness deficiency (not merely an upper bound for its randomness deficiency).  We show that both classes are different by identifying a function that is exactly layerwise computable but not layerwise computable.
 
Brattka, Gherardi and H\"olzl~\cite{BrHoeGh13} define and study the Weihrauch degree of $\LAY$, a function representing the mathematical task of determining an upper bound for the randomness deficiency of a given $\MLR$ sequence.  In particular, they investigate how $\LAY$ interacts with $\mathsf{MLR}$---the principle that generates sequences that are ML-random relative to its input---and the principle $\C_\N$---the choice principle on natural numbers.  We continue the study of $\LAY$ in Section~\ref{sec-Weihrauch}, where we show that, unlike the notion of layerwise computability, the Weihrauch degree of $\LAY$ does not depend on the choice of the universal ML-test used to define it.  Moreover, we show that, up to Weihrauch degree, the problem of exactly determining a ML-random sequence's randomness deficiency is equivalent to merely determining an upper bound for its randomness deficiency.  We show that the Weihrauch degree of $\LAY$ enjoys several idempotence properties, and we investigate the complexity of sets that can be reduced to $\LAY$.

Finally, in Section~\ref{sec-LWCvsLAY}, we compare layerwise computability with Weihrauch-reducibility to $\LAY$.  The results express that Weihrauch-reducibility to $\LAY$ encompasses a wider class of functions than layerwise computability.

Several of the results in this article, particularly in Section~\ref{sec-Weihrauch} and Section~\ref{sec-LWCvsLAY}, were obtained independently by Davie, Fouch\'e, and Pauly~\cite{PaulyPC}.

\section{Preliminaries}

\subsection{Computability theory and algorithmic randomness}
We typically follow Downey and Hirschfeldt~\cite{DH10} and Nies~\cite{Nies:book}, the standard references for algorithmic randomness.  $2^\omega$ is Cantor space, $2^s$ is the set of binary strings of length $s$, and $2^{<\omega} = \bigcup_{s \in \omega}2^s$ is the set of finite binary strings.  For $\sigma, \tau \in 2^{<\omega}$, we write $\sigma \preceq \tau$ if $\tau$ is an extension of $\sigma$.  We write $\sigma \prec \tau$ if this extension is strict.  For $\sigma \in 2^{<\omega}$ and $X \in 2^\omega$, we write $\sigma \prec X$ if $\sigma$ is an initial segment of $X$.  For $\sigma \in 2^{<\omega}$, $[\sigma]$ is the set $\{X \in 2^\omega : \sigma \prec X\}$.  For a set $U \subseteq 2^{<\omega}$, $[U] = \bigcup_{\sigma \in U}[\sigma]$ is the open subset of $2^\omega$ determined by $U$.  For a tree $T\subseteq 2^{<\omega}$, $[T]$ is the set of infinite paths through $T$.  One makes the analogous definitions concerning the topology on Baire space (i.e., on $\omega^\omega$), though our topological needs are almost exclusively confined to~$2^\omega$.  Lebesgue measure on Cantor space is denoted by $\lambda$.  As usual, $\langle \cdot, \cdot \rangle\colon \omega^2 \rightarrow \omega$ denotes the canonical computable and computably invertible pairing function.  We also use $\langle \cdot, \cdot \rangle$ to denote pairs of elements from $2^\omega$ or $\omega^\omega$, coded in the usual way.

$(\Phi_i)_{i \in \omega}$ is an effective list of all Turing functionals.  For a Turing functional $\Phi$, $\Phi(f)(n)$ denotes the output of $\Phi$ on input $n$ when equipped with oracle $f$, should the computation converge.  We suppress the `$(n)$' and write $\Phi(f)$ to denote the function computed by $\Phi$ when equipped with oracle $f$.  We sometimes suppress the `$(f)$' and write $\Phi(n)$ when $f$ is understood to be the identically $0$ function.  Subsets of $\omega$ are identified with their characteristic functions, and, for the purposes of supplying a single number as an oracle, each $n \in \omega$ is identified with the corresponding singleton $\{n\}$.

An open set $\MP{U}_0 \subseteq 2^\omega$ is \emph{effectively open} (or  \emph{c.e.})\ if the set $\{\sigma \in 2^{<\omega} : [\sigma] \subseteq \MP{U}_0\}$ is c.e.  A sequence of open sets $(\MP{U}_i)_{i \in \omega}$ is \emph{uniformly effectively open} (or \emph{uniformly c.e.})\ if the set $\{\la i, \sigma \ra : i \in \omega \andd \sigma \in 2^{<\omega} \andd [\sigma] \subseteq \MP{U}_i\}$ is c.e.

\begin{Definition}{\ }
\begin{itemize}
\item A \emph{Martin-L\"of test} is a uniformly c.e.\ sequence $\MP V = (\MP{V}_i)_{i \in \omega}$ of subsets of $2^\omega$ such that $\lambda(\MP{V}_i)\leq 2^{-i}$ for all $i$.

\item A Martin-L\"of test $\MP V$ is \emph{nested} if $\forall i(\MP{V}_{i+1} \subseteq \MP{V}_i)$.

\item A Martin-L\"of test $\MP U$ is \emph{universal} if, for every Martin-L\"of test $\MP V$, $\bigcap_{i \in \omega} \MP{V}_i \subseteq \bigcap_{i \in \omega} \MP{U}_i$.

\item A Martin-L\"of test $\MP U$ is \emph{optimal} if, for every Martin-L\"of test $\MP V$, $\exists c \forall i (\MP{V}_{i+c} \subseteq \MP{U}_i)$.

\item A sequence $X \in 2^\omega$ is \emph{Martin-L\"of random} (written $X \in \MLR$) if there is no ML-test $\MP V$ with $X \in \bigcap_{i \in \omega} \MP{V}_i$.
\end{itemize}
\end{Definition}

Every optimal ML-test is universal, and if $\MP U$ is a universal ML-test, then $X \in 2^\omega$ is ML-random if and only if $X \notin \bigcap_{i \in \omega}\MP{U}_i$.  Martin-L\"of~\cite{martin1966definition} defined the notions of ML-test and ML-randomness and proved that universal, indeed optimal, ML-tests exist (see also~\cite[Theorem 6.2.5]{DH10}).  In the same paper, Martin-L\"of also introduced the concept of \emph{randomness deficiency}, which can be thought of as a measure of how long it takes before a random sequence actually begins to look random.  That is, since the randomness of a sequence according to Martin-Löf's definition does not change if a finite initial segment of that sequence is changed, there are random sequences which do look extremely non-random on long initial segments, e.g., they might begin with $0^n$ for large $n$. The intuition is that the greater the randomness deficiency, the longer is this finite initial segment of non-randomness.

\begin{Definition}
If $X \in \MLR$ and $\MP U$ is a universal ML-test, then the \emph{randomness deficiency of $X$ relative to $\MP U$} is $\rd_{\MP U}(X) = \min\{i : X \notin \MP{U}_i\}$.
\end{Definition}
When the universal ML-test $\MP U$ is clear from context, we sometimes write $\rd$ in place of $\rd_{\MP U}$.

In a few places we make use of effective reals.  The following material and much more can be found in \cite[Chapter~5]{DH10}.

\begin{Definition}
Let $\alpha \in [0,1]$.  Then
\begin{itemize}
\item $\alpha$ is \emph{computable} if $\{q \in \mathbb{Q} : q < \alpha\}$ is computable,
\item $\alpha$ is \emph{left-c.e.}\ if $\{q \in \mathbb{Q} : q < \alpha\}$ is c.e., and
\item $\alpha$ is \emph{right-c.e.}\ if $\{q \in \mathbb{Q} : q > \alpha\}$ is c.e.
\end{itemize}
An $\alpha \in [0,1]$ is computable if and only if it is both left-c.e.\ and right-c.e.\ if and only if there is a computable sequence of rationals $(q_n)_{n \in \omega}$ such that $\forall n(|\alpha - q_n| < 2^{-n})$.  An $\alpha \in [0,1]$ is left-c.e.\ if and only if there is a computable sequence of rationals that is increasing and converges to $\alpha$.  An $\alpha \in [0,1]$ is right-c.e.\ if and only if $1-\alpha$ is left-c.e. if and only if there is a computable sequence of rationals that is decreasing and converges to $\alpha$.  The measure of an effectively open subset of $2^\omega$ is a left-c.e.\ real, and the measure of its complement in $2^\omega$ is a right-c.e.\ real.

\end{Definition}

\subsection{Layerwise computability}
Hoyrup and Rojas~\cite{MR2545900} introduced the concept of \emph{layerwise computability} which captures functions that, while not necessarily computable, are computable on $\MLR$ when given advice that essentially corresponds to an upper bound for the randomness deficiency of the input sequence.  They then employ this notion in the study of Birkhoff's ergodic theorem and Brownian motion, to give two examples~\cites{hoyrup2013computability, hoyrup2009applications}.  Hoyrup and Rojas define layerwise computability in a general computable probability space, but we present the definition only for Lebesgue measure on Cantor space for simplicity and because the phenomena we wish to study are present in this context.

\begin{Definition}[Hoyrup and Rojas% 
%[Definition~3.2.1]
~\cite{MR2545900}] 
A set $\MP A \subseteq 2^\omega$ is \emph{effectively $\lambda$-measurable} if there are sequences of uniformly effectively open sets $(\MP{U}_i)_{i \in \omega}$ and $(\MP{V}_i)_{i \in \omega}$ such that $2^\omega \setminus \MP{V}_i \subseteq \MP A \subseteq \MP{U}_i$ and $\lambda(\MP{U}_i \cap \MP{V}_i) \leq 2^{-i}$ for all $i \in \omega$.
\end{Definition}

\begin{Definition}[Hoyrup and Rojas \cite{MR2545900}]
Let $\MP W$ be a universal ML-test.
\begin{itemize}
\item A set $\MP A \subseteq 2^\omega$ is \emph{$\MP W$-layerwise semi-decidable} if there is a sequence of uniformly effectively open sets $(\MP{U}_i)_{i \in \omega}$ such that $\MP A \cap (2^\omega \setminus \MP{W}_i) = \MP{U}_i \cap (2^\omega \setminus \MP{W}_i)$ for all $i \in \omega$.

\item A set $\MP A \subseteq 2^\omega$ is \emph{$\MP W$-layerwise decidable} if both $\MP A$ and $2^\omega \setminus \MP A$ are $\MP W$-layerwise semi-decidable.

\item A function $f \colon 2^\omega \imp 2^\omega$ is \emph{$\MP W$-layerwise computable} if there is a Turing functional $\Phi$ such that $(\forall i)(\forall X \in 2^\omega \setminus \MP{W}_i)(\Phi(\la X,i \ra) = f(X))$.  
\end{itemize}
\end{Definition}

\begin{Remark}
When considering functions $f \colon 2^\omega \imp \omega$ (such as characteristic functions of subsets of $2^\omega$ or randomness deficiency functions), it is equivalent to define $f$ as being $\MP W$-layerwise computable if there is a Turing functional $\Phi$ such that $(\forall i)(\forall X \in 2^\omega \setminus \MP{W}_i)(\Phi(X)(i) = f(X))$. 
\end{Remark}

For a universal ML-test $\MP W$, it is straightforward to show that a set $\MP A \subseteq 2^\omega$ is $\MP W$-layerwise decidable if and only if its characteristic function is $\MP W$-layerwise computable.  Furthermore, effective $\lambda$-measurability and layerwise decidability coincide for optimal ML-tests.
\begin{Theorem}[Hoyrup and Rojas~\cite{MR2545900}%
%~Theorem~4.2.1(2)
]\label{thm-H&R}
Let $\MP U$ be an optimal ML-test, and let $\MP A \subseteq 2^\omega$.  Then $\MP A$ is effectively $\lambda$-measurable if and only if $\MP A$ is $\MP U$-layerwise decidable.
\end{Theorem}

\subsection{The Weihrauch degrees}

We refer the reader to Weihrauch~\cite{weihrauch2000computable} for a general reference for computable analysis.  The Weihrauch degrees are the degree structure induced by Weihrauch reducibility.  This reducibility was defined by Klaus Weihrauch and has been the subject of intense study \cites{Ste89,Wei92a,weihrauch1992tte,hertling1996unstetigkeitsgrade,MR1650860,MALQ:MALQ200310125}.  Recently, it has been used to classify the computational difficulty of solving specific mathematical tasks~\cites{gherardi2009,MALQ:MALQ200910104,Pauly:jucs_16_18:how_incomputable_is_finding,MR2791341,MR2760117,MR2915694,MR2889550,FrancoisDamirJeffJoePaul}.  The work in this article was sparked by the work in Brattka, Gherardi, and H\"olzl~\cite{BrHoeGh13}.

Informally, assume that we have a black box that is given some input and produces some output.  In all cases but the trivial ones, the transformations performed by the black boxes are non-computable.  In full generality, inputs and outputs can be members of any two so-called \emph{represented spaces}, but in this article, an input is typically a ML-random sequence, and an output is typically a single natural number bounding the randomness deficiency of the input sequence.

We think of the transformation performed by a black box as solving some mathematical problem $A$.  That is, the input is an instance of $A$, and the output a solution to that instance.  We can then, given such a box, investigate whether we can also use it to solve some other mathematical problem $B$.  That is, we ask if it is possible to encode an instance of $B$ into an instance of $A$, obtain a solution to this $A$-instance using the black box, and finally decode the solution to the $A$-instance into a solution to the original $B$-instance.  This informal notion of reducibility can be seen as a sort of many-one reducibility among mathematical problems.  We now give the formal definitions pertaining to the Weihrauch degrees.

\begin{Definition}
A \emph{representation} of a set $X$ is a surjective partial function $\delta\colon \!\!\subseteq\omega^\omega\to X$.  Call $(X, \delta)$ a \emph{represented space}.
\end{Definition}

The `$\subseteq$' in `$\delta\colon \!\!\subseteq\omega^\omega\to X$' emphasizes that the domain of $\delta$ may be a proper subset of~$\omega^\omega$.  Sometimes we suppress the representation and write $X$ in place of $(X, \delta)$ for the sake of readability.

\begin{Definition}
Let $f\colon \!\!\subseteq(X, \delta_X) \rightrightarrows (Y, \delta_Y)$ be a multi-valued partial function on represented spaces.  A function $\Gamma\colon \!\!\subseteq\omega^\omega\to\omega^\omega$ \emph{realizes} $f$ (or is a \emph{realizer} of $f$, written $\Gamma \vdash f$) if $\delta_Y \circ \Gamma(p)\in f \circ \delta_X(p)$ for all $p\in\dom(f\circ \delta_X)$.
\end{Definition}

The `$\rightrightarrows$' in `$f\colon \!\!\subseteq(X, \delta_X) \rightrightarrows (Y, \delta_Y)$' emphasizes that the function is multi-valued.  The intuition behind these two definitions is that a representation allows us to use elements of Baire space as codes for elements of other spaces and that a function between two such other spaces can be represented as a realizer (that is, as a function from codes for elements of the input space to codes for the corresponding elements of the output space).

\begin{Definition}
Let $f$ and $g$ be two multi-valued functions on represented spaces.  Then $f$ is \emph{Weihrauch reducible} to $g$ ($f\leqW g$) if there are Turing functionals $\Phi,\Psi\colon \!\!\subseteq\omega^\omega \to \omega^\omega$ such that $\Psi \circ \langle \id, \Gamma\circ \Phi \rangle \vdash f$ for all $\Gamma \vdash g$.
\end{Definition}

The intuition is that $\Phi$ pre-processes an input before handing it over to $\Gamma$ and that $\Psi$ post-processes the output provided by $\Gamma$.  In terms of the informal reducibility described in terms of black boxes above, $\Gamma$ is the black box, $\Phi$ performs the encoding, and $\Psi$ performs the decoding.  Note that $\Psi$ has access to the original input.
\begin{Definition}
Let $f$ and $g$ be two multi-valued functions on represented spaces.  Then $f$ is \emph{strongly Weihrauch reducible} to $g$ ($f \leqsW g$) if there are Turing functionals $\Phi,\Psi\colon \!\!\subseteq\omega^\omega \to \omega^\omega$ such that $\Psi \circ \Gamma\circ \Phi \vdash f$ for all $\Gamma \vdash g$.
\end{Definition}

Note that in strong Weihrauch reductions, $\Psi$ does not have access to the original input.

Although we do not study the lattice structure of these degrees in this work, we mention that the Weihrauch degrees form a distributive lattice by results by Brattka and Pauly~\cites{MR2791341,MALQ:MALQ200910104} and that the strong Weihrauch degrees form a lower semi-lattice by Brattka~\cite{MR2791341}; whether the latter actually form a lattice is unknown.

Finally, we introduce the operations on multi-valued functions that we study in Section~\ref{sec-Weihrauch}.

\begin{Definition}
Let $f \colon \!\!\subseteq X \rightrightarrows Y$ and $g \colon \!\!\subseteq W \rightrightarrows Z$ be multi-valued functions.  The \emph{product} of $f$ and $g$ is the multi-valued function $f \times g \colon \!\!\subseteq X \times W \rightrightarrows Y \times Z$ defined by $(f \times g)(x,w) = f(x) \times g(w)$.
\end{Definition}

The \emph{parallelization} of a function $f$ is the infinite product of $f$ with itself.

\begin{Definition}
Let $f\colon \!\!\subseteq X \rightrightarrows Y$ be a multi-valued function.  The \emph{parallelization} of $f$ is the multi-valued function $\widehat{f} \colon \!\!\subseteq X^\omega \imp Y^\omega$ defined by $\widehat{f}((x_i)_{i \in \omega}) = (f(x_i))_{i \in \omega}$.
\end{Definition}

The \emph{compositional product} of two Weihrauch degrees was introduced by Brattka, Gherardi, and Marcone~\cite{MR2889550}.  The intuition is that the computational product of $f$ and $g$ describes the maximum complexity of a function that can be computed by calling $g$, making an additional computation, and then calling $f$.

\begin{Definition}
Let $f$ and $g$ be two multi-valued functions.  The \emph{compositional product} of $f$ and $g$ is the Weihrauch degree 
\begin{align*}
f \ast g = \sup\{f_0 \circ g_0 : (f_0 \leqW f) \andd (g_0 \leqW g) \andd (\text{$f$ and $g$ are composable)}\},
\end{align*}
where the supremum is with respect to $\leqW$.
\end{Definition}

Brattka, Oliva, and Pauly~\cite{BOP13} prove that $f \ast g$ is well-defined on the Weihrauch degrees and that the supremum is in fact a maximum.  Trivially, $f \times g \leqW f \ast g$.  Brattka, Gherardi, and Marcone~\cite{MR2889550} also define the \emph{strong compositional product} in the strong Weihrauch degrees by replacing $\leqW$ with $\leqsW$ and taking the supremum with respect to $\leqsW$.  However, strong compositional products are not guaranteed to exist in general.

\section{Non-optimal universal Martin-L\"of tests}\label{sec-NonOptimal}

Every optimal ML-test is universal.  In fact, if $\MP U$ is a ML-test such that, for every ML-test $\MP V$, $\forall i \exists j (\MP{V}_j \subseteq \MP{U}_i)$, then $\MP U$ is universal.  However, there are a universal ML-test $\MP U$ and an ML-test $\MP V$ such that $\exists i \forall j (\MP{V}_j \nsubseteq \MP{U}_i)$.

\begin{Lemma}\label{lem-NotOptimal}
There are an effectively open set $\MP{W}_0$ and a ML-test $\MP V$ such that $\overline\MLR \subseteq \MP{W}_0$ but $\forall i(\MP{V}_i \nsubseteq \MP{W}_0)$.
\end{Lemma}

\begin{proof}
Let $\MP U$ be a universal ML-test, and let $(\MP{U}_{2,s})_{s \in \omega}$ be an enumeration of the basic open sets in $\MP{U}_2$.  Define a sequence $(\sigma_s)_{s \in \omega}$ by letting each $\sigma_s \in 2^{<\omega}$ be the first string of length $\geq s+2$ such that $[\sigma_s]$ is disjoint from $\MP{U}_{2,s} \cup \bigcup_{t<s}[\sigma_t]$.  Enumerate in stages the effectively open set $\MP{W}_0'$ by $\MP{W}_{0,s}' = \MP{U}_{2,s} \setminus \bigcup_{t<s}[\sigma_t]$.  By choice of the sequence $(\sigma_s)_{s \in \omega}$, the result is $\MP{W}_0' = \MP{U}_2 \setminus \bigcup_{s \in \omega}[\sigma_s]$.  Let $\MP{W}_0 = \MP{W}_0' \cup \bigcup_{s \in \omega}([\sigma_s] \cap \MP{U}_{|\sigma_s|+1})$ and, for each $i \in \omega$, let $\MP{V}_i = [\sigma_i]$.  $\MP V$ is a ML-test because $\lambda(\MP{V}_i) = \lambda([\sigma_i]) = 2^{-|\sigma_i|} \leq 2^{-(i+2)}$.  To see that $\overline\MLR \subseteq \MP{W}_0$, consider $X \in \overline\MLR = \bigcap_{i \in \omega} \MP{U}_i$.  If $X \in [\sigma_s]$ for some $s \in \omega$, then $X \in [\sigma_s] \cap \MP{U}_{|\sigma_s|+1} \subseteq \MP{W}_0$.  If not, then $X \in \MP{W}_0' \subseteq \MP{W}_0$.  In either case, $X \in \MP{W}_0$, as desired.  To see that $\forall i(\MP{V}_i \nsubseteq \MP{W}_0)$, observe that $\MP{V}_i = [\sigma_i]$ and that $[\sigma_i] \cap \MP{W}_0 = [\sigma_i] \cap \MP{U}_{|\sigma_i|+1}$.  Thus
\begin{align*}
\lambda(\MP{V}_i \cap \MP{W}_0) \leq \lambda(\MP{U}_{|\sigma_i|+1}) \leq 2^{-(|\sigma_i|+1)} < 2^{-|\sigma_i|} = \lambda(\MP{V}_i),
\end{align*}
so $\MP{V}_i \nsubseteq \MP{W}_0$.
\end{proof}

\begin{Theorem}\label{thm-NotOptimal}
There are ML-tests $\MP U$ and $\MP V$ such that $\MP U$ is universal but $\exists i \forall j (\MP{V}_j \nsubseteq \MP{U}_i)$.
\end{Theorem}

\begin{proof}
Let $\MP{W}_0$ and $\MP V$ be as in Lemma~\ref{lem-NotOptimal}, let $\MP U$ be any universal ML-test, and replace $\MP{U}_0$ by~$\MP{W}_0$.  $\MP U$ remains universal because $\overline\MLR \subseteq \MP{W}_0$, and $i=0$ witnesses $\exists i \forall j (\MP{V}_j \nsubseteq \MP{U}_i)$.
\end{proof}

\begin{Remark}
The ML-test $\MP V$ constructed in the proof of Lemma~\ref{lem-NotOptimal} satisfies $\bigcap_{i \in \omega}\MP{V}_i = \emptyset$.  This is not necessary because we could let $\MP{V}'$ be any ML-test and replace each $\MP{V}_i$ by $\MP{V}_{i+1} \cup \MP{V}'_{i+1}$.  Hence the ML-test $\MP{V}$ in Lemma~\ref{lem-NotOptimal} and Theorem~\ref{thm-NotOptimal} may be taken to be universal or even optimal.
\end{Remark}

On the other hand, if $\MP U$ and $\MP V$ are ML-tests such that $\forall i \exists j (\MP{V}_j \subseteq \MP{U}_i)$, then there is a function $f \leqT 0''$ such that $\forall i(\MP{V}_{f(i)} \subseteq \MP{U}_i)$.  The next theorem states that, in this situation, $0''$ is necessary in general.

\begin{Theorem}\label{thm-ComputeTot}
There are ML-tests $\MP W$ and $\MP V$ such that
\begin{enumerate}[(i)]
\item $\MP W$ is universal,
\item $\forall i \exists j (\MP{V}_j \subseteq \MP{W}_i)$, and
\item if $f \colon \omega \imp \omega$ satisfies $\forall i (\MP{V}_{f(i)} \subseteq \MP{W}_i)$, then $f \geqT 0''$.\label{total_thm3}
\end{enumerate}
\end{Theorem}

\begin{proof}
By the following (well-known) claim, it suffices to construct $\MP W$ and $\MP V$ such that every $f$ satisfying $\forall i (\MP{V}_{f(i)} \subseteq \MP{W}_i)$ also satisfies 
\begin{align*}
\forall e(\exists n \Phi_e(n)\ua \imp (\exists n \leq f(e))\Phi_e(n)\ua)\tag{$\dagger$}\label{total_thm}.
\end{align*}

\begin{Claim}
Suppose $f \colon \omega \imp \omega$ satisfies \textnormal{(\ref{total_thm})}.  Then $f \geqT 0''$.
\end{Claim}
\begin{proof}[Proof of claim]
First we show that $f \geqT 0'$.  To determine whether or not $\Phi_e(e)\da$, generate an index $i$ such that $\forall n(\Phi_i(n)\da \biimp \Phi_{e,n}(e)\ua)$.  Then $\Phi_e(e)\da$ if and only if $(\exists n \leq f(i))\Phi_{e,n}(e)\da$.  This shows that $f \geqT 0'$.  Now we readily see that $f \geqT 0''$ as well because $\Phi_e$ is total if and only if $(\forall n \leq f(e))\Phi_e(n)\da$, and this latter predicate is $f$-computable because $f \geqT 0'$.
\end{proof}

Let $\MP U$ be an optimal ML-test.  We simultaneously enumerate $\MP W$ and $\MP V$ in stages.  For each $i$, initialize $n_{i,0} = 0$ and $e_{i,0} = i+1$.  At stage $\la i,s \ra$, proceed as follows.
\begin{itemize}
\item If $\Phi_{i,s}(n_{i,s})\ua$, then enumerate $\MP{U}_{e_{i,s},s}$ into $\MP{W}_i$, set $n_{i,s+1} = n_{i,s}$, and set $e_{i,s+1} = e_{i,s}$.
\item If $\Phi_{i,s}(n_{i,s})\da$, then enumerate $\MP{U}_{i+1,s}$ into $\MP{W}_i$.  Let $\sigma$ be the first string such that $[\sigma]$ is disjoint from $\MP{W}_i$ and that $(\forall j \leq n_{i,s})(\lambda(\MP{V}_j \cup [\sigma]) < 2^{-j})$.  Enumerate $[\sigma]$ into $\MP{V}_j$ for each $j \leq n_{i,s}$, set $n_{i,s+1} = n_{i,s}+1$, and set $e_{i,s+1} = \max(e_{i,s}, |\sigma|)+1$.
\end{itemize}

It is easy to see that $\forall i(\lambda(\MP{V}_i) \leq 2^{-i})$.  To see that $\forall i (\lambda(\MP{W}_i) \leq 2^{-i})$, fix $i$ and observe that $\MP{W}_i \subseteq \bigcup_{e \geq i+1} \MP{U}_e$ and hence that $\lambda(\MP{W}_i) \leq \sum_{e \geq i+1}\lambda(\MP{U}_e) \leq \sum_{e \geq i+1} 2^{-e} = 2^{-i}$.  Thus $(\MP{W}_i)_{i \in \omega}$ and $(\MP{V}_i)_{i \in \omega}$ are indeed ML-tests.  Furthermore, we have that $\forall i \exists j(\MP{U}_j \subseteq \MP{W}_i)$.  Fix $i$.  If $\exists n \Phi_i(n)\ua$, then $n_i = \lim_s n_{i,s}$ exists and is the least such $n$, $e_i = \lim_s e_{i,s}$ exists, and $\MP{U}_{e_i} \subseteq \MP{W}_i$.  If $\forall n \Phi_i(n)\da$, then $\MP{U}_{i+1} \subseteq \MP{W}_i$.  The fact that $\MP U$ is an optimal ML-test now immediately implies that $\MP W$ is universal and that $\forall i \exists j (\MP{V}_j \subseteq \MP{W}_i)$.

To finish the proof, it remains to show (\ref{total_thm3}). For this we argue that each $f$ as in (\ref{total_thm3}) has property (\ref{total_thm}). To see this, it suffices to show, for each $i$, that if $n$ is least such that $\Phi_i(n)\ua$, then $(\forall j < n)(\MP{V}_j \nsubseteq \MP{W}_i)$.  If this is the case, then any $f$ with $\forall i (\MP{V}_{f(i)} \subseteq \MP{W}_i)$ must have $f(i)\geq n$, thereby bounding the first place of divergence $n$ of $\Phi_i$ as required by (\ref{total_thm}).

So fix $i$, suppose that $n$ is least such that $\Phi_i(n)\ua$, and assume $n > 0$ (for if $n=0$ the desired property trivially holds).  Then $n = \lim_s n_{i,s}$, and there is a stage $\la i,s \ra$ at which $n_{i,s} = n-1$ and $\Phi_{i,s}(n_{i,s})\da$.  At this stage we enumerate a $[\sigma]$, disjoint from what has been enumerated into $\MP{W}_i$ so far, into $\MP{V}_j$ for every $j \leq n-1$.  From this stage on, we enumerate $\MP{U}_{e_i}$ into $\MP{W}_i$, where $e_i = \lim_s e_{i,s} > |\sigma|$.  It follows that $[\sigma] \nsubseteq \MP{W}_i$ (because $\lambda([\sigma]) > \lambda([\sigma] \cap \MP{W}_i)$) and therefore that $(\forall j < n)(\MP{V}_j \nsubseteq \MP{W}_i)$.
\end{proof}

\begin{Remark}
In fact, the universal ML-test $\MP W$ constructed in Theorem~\ref{thm-ComputeTot} has the following properties.
\begin{itemize}
\item If $\MP U$ is any ML-test, then $\forall i \exists j(\MP{U}_j \subseteq \MP{W}_i)$.

\item If $\MP U$ is an optimal ML-test and $f \colon \omega \imp \omega$ is a function such that $\forall i (\MP{U}_{f(i)} \subseteq \MP{W}_i)$, then $f \geqT 0''$.
\end{itemize}
\end{Remark}

\section{Layerwise computability depends on the ML-test}\label{sec-LWC}

Let $\MP U$ and $\MP V$ be universal ML-tests.  If there is a computable function $f \colon \omega \rightarrow \omega$ such that $\forall i (\MP{V}_{f(i)} \subseteq \MP{U}_i)$, then every $\MP V$-layerwise computable function is also $\MP U$-layerwise computable.  In particular, if $\MP U$ is an optimal ML-test, then every $\MP V$-layerwise computable function is $\MP U$-layerwise computable; and if $\MP U$ and $\MP V$ are both optimal ML-tests, then a function is $\MP U$-layerwise computable if and only if it is $\MP V$-layerwise computable.  However, the results of the previous section show that in general a function $f \colon \omega \rightarrow \omega$ such that $\forall i (\MP{V}_{f(i)} \subseteq \MP{U}_i)$ need not exist at all.  Moreover, there are cases in which such functions do exist but are necessarily difficult to compute.

This situation raises the question of whether or not there are a universal ML-test $\MP U$ and a (necessarily non-optimal) universal ML-test $\MP V$ such that some function is $\MP U$-layerwise computable but not $\MP V$-layerwise computable.  We show that such ML-tests indeed exist.  Let $\MP U$ be an optimal ML-test.  Recall from Theorem~\ref{thm-H&R} that a set is effectively $\lambda$-measurable if and only if it is $\MP U$-layerwise decidable. We will show that if a universal ML-test $\MP V$ is not optimal then  effectively $\lambda$-measurable sets need not be $\MP V$-layerwise decidable.

\begin{Theorem}\label{thm-BadLayering}
There are an effectively $\lambda$-measurable set $\MP A \subseteq 2^\omega$ and a universal ML-test $\MP W$ such that $\MP A$ is not $\MP W$-layerwise decidable.
\end{Theorem}

\begin{proof}
 The idea of the proof is to build $\MP A$ in such a way that no functional $\Phi_i$ is a $\MP W$-layerwise computation procedure for the characteristic function of $\MP A$.  We achieve this by waiting for $\Phi_i(\cdot)(i)$ to converge to either $0$ or $1$ on some $[\sigma]$.  In the first case we put $[\sigma]$ into $\MP A$, and in the second case we ensure that $[\sigma]$ is never put into $\MP A$.  Then $[\sigma]$ witnesses the failure of $\Phi_i$.
 
We simultaneously enumerate the ML-test $(\MP{W}_i)_{i \in \omega}$ and compute two sequences $(\IN_s)_{s \in \omega}$ and $(\OUT_s)_{s \in \omega}$ of finite sets of strings from which we define the effectively $\lambda$-measurable set~$\MP A$.  In fact, $\MP A$ will be an effectively open set. Throughout the construction, we maintain that $\lambda(\MP{W}_i) \leq 2^{-i-4}$ for all $i \in \omega$ and that $\lambda([\IN_s]) \leq 1/16$, $\lambda([\OUT_s]) \leq 1/16$, and $[\IN_s] \cap [\OUT_s] = \emptyset$ for all $s \in \omega$.

Let $\MP Y$ be any nested universal ML-test.  Initialize $\MP{W}_{i,0} = \emptyset$ for each $i \in \omega$, $\IN_0 = \OUT_0 = \emptyset$, and $e_{i,0} = i+4$ for each $i \in \omega$.  At stage $s = \la i,t \ra$, check if $t$ is least such that there is an $n\in \omega$ and and a sequence of strings $(\sigma_j)_{j < n}$ from $2^t$ such that $\lambda(\bigcup_{j<n}[\sigma_j]) \geq 1/2$ and $(\forall j < n)(\Phi_i(\sigma_j)(i)\da < 2)$.  If $t$ is indeed least, search for the first $\sigma \in 2^{<\omega}$ such that $\Phi_i(\sigma)(i)\da < 2$, $[\sigma] \cap (\MP{W}_{i,t} \cup [\IN_s] \cup [\OUT_s]) = \emptyset$, and $\lambda([\sigma]) \leq 2^{-s-5}$.  Such a $\sigma$ must exist because the measures of $\MP{W}_{i,t}$, $[\IN_s]$, and $[\OUT_s]$ are all $\leq 1/16$.  Set $e_{i,s+1} = \max(e_{i,s}, |\sigma|)+1$.  If $\Phi_i(\sigma)(i)=0$, set $\IN_{s+1} = \IN_s \cup \{\sigma\}$ and $\OUT_{s+1} = \OUT_s$.  If $\Phi_i(\sigma)(i)=1$, set $\IN_{s+1} = \IN_s$ and $\OUT_{s+1} = \OUT_s \cup \{\sigma\}$.  If no sequence $(\sigma_j)_{j < n}$ as above was found for this $t$ or if $t$ is not least, set $e_{i,s+1} = e_{i,s}$, $\IN_{s+1} = \IN_s$, and $\OUT_{s+1} = \OUT_s$.  In either case, put $\MP{W}_{i,t+1} = \MP{W}_{i,t} \cup \MP{Y}_{e_{i,t+1},t}$.

The above procedure defines the ML-test $\MP W$ and the sequences $(\IN_s)_{s \in \omega}$ and $(\OUT_s)_{s \in \omega}$.  $\MP W$ is indeed an ML-test because (by the nestedness of $\MP Y$) $\forall i (\MP{W}_i \subseteq \MP{Y}_{i+4})$.  Hence $\forall i(\lambda(\MP{W}_i) \leq \lambda(\MP{Y}_{i+4}) \leq 2^{-(i+4)})$.  To see that $\MP W$ is universal, it suffices to notice that, for all $i$, the limit $e_i = \lim_s e_{i,s}$ exists because $e_{i,s} \neq e_{i,s+1}$ for at most one $s$.  This implies that $\exists j(\MP{Y}_j \subseteq \MP{W}_i)$, from which the universality of $\MP W$ follows from the universality of $\MP Y$.

Let $\MP A = \bigcup_{s \in \omega}[\IN_s]$.  We show that $\MP A$ is effectively $\lambda$-measurable.  Let $\MP{U}_i = \MP A$ for each $i \in \omega$, and let $\MP{V}_i = 2^\omega \setminus [\IN_i]$ for each $i \in \omega$.  Then $\MP{U}_i \cap \MP{V}_i = (\bigcup_{s \in \omega}[\IN_s]) \setminus [\IN_i]$.  At any stage $s$, at most one $\sigma$ enters $\IN_s$, and such a $\sigma$ satisfies $\lambda([\sigma]) \leq 2^{-s-5}$.  Thus $\lambda(\MP{U}_i \cap \MP{V}_i) \leq \sum_{s>i}2^{-s-5} = 2^{-i-5} \leq 2^{-i}$, so $(\MP{U}_s)_{s \in \omega}$ and $(\MP{V}_s)_{s \in \omega}$ witness that $\MP A$ is effectively $\lambda$-measurable.  It is also easy to see that $[\IN_s]$ and $[\OUT_s]$ are disjoint for each $s \in \omega$.  That the measures of $[\IN_s]$ and $[\OUT_s]$ are $\leq 1/16$ for each $s \in \omega$ is because at each stage $s$ there is at most one $\sigma$ added to $\IN_s$ or $\OUT_s$, the corresponding $[\sigma]$ has measure at most $2^{-s-5}$, and $\sum_{s \geq 0}2^{-s-5} = 1/16$.

Finally, we show that $\MP A$ is not $\MP W$-layerwise decidable.  Suppose for a contradiction that $\Phi_i$ witnesses that $\MP A$ is $\MP W$-layerwise decidable, and let $\MP{K}_i = 2^\omega \setminus \MP{W}_i$.  Then $(\forall X \in \MP{K}_i \cap \MP A)(\Phi_i(X)(i)=1)$ and $(\forall X \in \MP{K}_i \setminus \MP A)(\Phi_i(X)(i)=0)$.  As $\lambda(\MP{K}_i) \geq 1-2^{-i-4} \geq 1/2$, there is a stage $s = \la i,t \ra$ where $t$ is least such that there is an $n\in \omega$ and a sequence of strings $(\sigma_j)_{j < n}$ from $2^t$ such that $\lambda(\bigcup_{j<n}[\sigma_j]) \geq 1/2$ and $(\forall j < n)(\Phi(\sigma_j)(i)\da < 2)$.  Thus at stage $s$ we found a string $\sigma$ such that either $\Phi_i(\sigma)(i)=0$, in which case we added $\sigma$ to $\IN_{s+1}$; or $\Phi_i(\sigma)(i)=1$, in which case we added $\sigma$ to $\OUT_{s+1}$.  Regardless of the choice of $\sigma$, we have that $[\sigma] \cap \MP{K}_i \neq \emptyset$ because $[\sigma]$ was chosen disjoint from $\MP{W}_{i,t}$, and a set of measure at most $\lambda([\sigma])/2$ was enumerated into $\MP{W}_i$ from stage $t+1$ on.  If $\Phi_i(\sigma)(i)=0$, then $\sigma \in \IN_{s+1}$, so $[\sigma] \subseteq \MP A$.  Then any $X \in [\sigma] \cap \MP{K}_i$ gives the contradiction $X \in \MP{K}_i \cap \MP A$ but $\Phi_i(X)(i)=0$.  On the other hand, if $\Phi_i(\sigma)(i)=1$, then $\sigma \in \OUT_{s+1}$, in which case $[\sigma] \subseteq 2^\omega \setminus \MP A$.  Then any $X \in [\sigma] \cap \MP{K}_i$ gives the contradiction $X \in \MP{K}_i \setminus \MP A$ but $\Phi_i(X)(i)=1$.  Therefore $\MP A$ is not $\MP W$-layerwise decidable.
\end{proof}

\begin{Corollary}\label{cor-LWCDependsOnTest}
The notion of layerwise computability depends on the universal ML-test.  That is, there are universal ML-tests $\MP U$ and $\MP W$ and a function $F$ that is $\MP U$-layerwise computable but not $\MP W$-layerwise computable.
\end{Corollary}

\begin{proof}
Let $\MP U$ be an optimal ML-test, let $\MP W$ and $\MP A$ be as in Theorem~\ref{thm-BadLayering}, and let $F$ be the characteristic function of $\MP A$.  $\MP A$ is effectively $\lambda$-measurable, therefore $\MP A$ is $\MP U$-layerwise decidable and $F$ is $\MP U$-layerwise computable.  However, $\MP A$ is not $\MP W$-layerwise decidable, so $F$ is not $\MP W$-layerwise computable.
\end{proof}

If $f$ is $\MP{U}$-layerwise computable via $\Phi$, then $\Phi(\la X, i \ra)$ is required to produce the correct value of $f(X)$ for \emph{every} pair $\la X, i \ra$ with $X \in 2^\omega \setminus \MP{U}_i$.  We show that more functions become $\MP U$-layerwise computable if the definition is relaxed to only require $\Phi(\la X, i \ra)=f(X)$ for the \emph{least} $i$ such that $X \in 2^\omega \setminus \MP{U}_i$.  This contrasts with the situation in the Weihrauch degrees, as we see in Section~\ref{sec-Weihrauch}.

\begin{Definition}
Let $\MP U$ be a universal ML-test.  A function $f \colon 2^\omega \imp 2^\omega$ is \emph{exactly $\MP U$-layerwise computable} if there is a Turing functional $\Phi$ such that $(\forall X \in \MLR)(\Phi(\la X, \rd_{\MP U}(X) \ra) = f(X))$.  
\end{Definition}

The function $\rd_{\MP U}(X)$ is clearly exactly $\MP U$-layerwise computable and thus is the natural candidate for a function that is exactly $\MP U$-layerwise computable but not $\MP U$-layerwise computable.  We show that $\rd_{\MP U}(X)$ is indeed not $\MP U$-layerwise computable, thus showing that the exactly $\MP U$-layerwise computable functions are always a strictly larger class than the $\MP U$-layerwise computable functions.  To implement the proof, we first show a basic lemma about summing uniformly right-c.e.\ sequences of reals.

\begin{Lemma}\label{lem-SumRightCE}
If $(r_n)_{n \in \omega}$ is a uniformly right-c.e.\ sequence of reals such that $\forall n(r_n \in [0, 2^{-(n+1)}])$, then $r = \sum_{n \in \omega} r_n$ is right-c.e.
\end{Lemma}

\begin{proof}
As any real $\alpha \in [0,1]$ is right-c.e.\ if and only if $1-\alpha$ is left-c.e., it suffices to show that $1-r$ is left-c.e.  We have that
\begin{align*}
1 - r = 1 - \sum_{n \in \omega} r_n = \sum_{n \in \omega} 2^{-(n+1)} - \sum_{n \in \omega} r_n = \sum_{n \in \omega}(2^{-(n+1)} - r_n).
\end{align*}
The sequence $(2^{-(n+1)} - r_n)_{n \in \omega}$ is uniformly left-c.e., and $\forall n (2^{-(n+1)} - r_n \in [0, 2^{-(n+1)}])$.  It is straightforward to show that $\sum_{n \in \omega} \alpha_n$ is left-c.e.\ whenever $(\alpha_n)_{n \in \omega}$ is a uniformly left-c.e.\ sequence of reals such that $\forall n(\alpha_n \in [0, 2^{-(n+1)}])$.  Thus, $1 - r$ is left-c.e.\end{proof}

\begin{Theorem}\label{thm-RDnotLWC}
Let $\MP U$ be a universal ML-test.  Then $\rd_{\MP U}$ is not $\MP U$-layerwise computable.
\end{Theorem}

\begin{proof}
Fix a uniformly computable sequence of trees $(T_i)_{i \in \omega}$ such that $\forall i([T_i] = 2^\omega \setminus \MP{U}_i)$.  Suppose for a contradiction that $\rd_{\MP U}$ is layerwise computable, and let $\Phi$ be such that $\Phi(X)(i) = \rd_{\MP U}(X)$ whenever $X \in [T_i]$.  Define another uniformly computable sequence of trees $(S_i)_{i \in \omega}$ by
\begin{align*}
S_i = \{\sigma \in T_i : \Phi(\sigma)(i)\da \imp \Phi(\sigma)(i) = i\}.
\end{align*}  
We claim that $\forall i ([S_i] = \{X \in \MLR : \rd_{\MP U}(X) = i\})$.  To see this, fix $i \in \omega$ and consider an $X \in \MLR$.  If $\rd_{\MP U}(X) = i$, then $X \in [T_i]$ and, as $\Phi(X)(i) = \rd_{\MP U}(X) = i$, it must be that $\forall n(\Phi(X \restriction n)(i)\da \imp \Phi(X \restriction n)(i) = i)$.  Hence $X \in [S_i]$.  If $\rd_{\MP U}(X) \neq i$, then either $X \notin [T_i]$, or $X \in [T_i]$ and $\Phi(X)(i) = \rd_{\MP U}(X) \neq i$.  If $X \notin [T_i]$, then clearly $X \notin [S_i]$.  If $\Phi(X)(i) = \rd_{\MP U}(X) \neq i$, then there is an $n \in \omega$ such that $\Phi(X \restriction n)(i)\da$ but $\Phi(X \restriction n)(i) \neq i$.  Thus $X \notin [S_i]$.

Let $S = S_0 \cup S_1$.  Note that $[S]$ is non-empty because there must be an $X \in \MLR$ such that $\rd_{\MP U}(X)$ is either $0$ or $1$. This is because
since $\mu(\MP U_1) \leq 1/2$, there is some $X \not \in \MP U_1$. Then either $\MP U_0 \not = 2^\omega$, in which case there is an $X$ with $\rd_{\MP U}(X)=0$, or $\MP U_0 = 2^\omega$, in which case there is an $X$ with $\rd_{\MP U}(X)=1$. 

  We show that $\lambda([S])$ is a computable real, which contradicts the fact that the measure of every non-empty $\Pi^0_1$ class of ML-random sets is a ML-random real (see, for example, Nies~\cite[Theorem~3.2.35]{Nies:book}).  To do this, we show that both $\lambda([S])$ and $1-\lambda([S])$ are right-c.e.  It is clear that $\lambda([S])$ is right-c.e.\ because it is the measure of a $\Pi^0_1$ class.  To see that $1-\lambda([S])$ is right-c.e., observe that $([S_i])_{i \in \omega}$ is a partition of $\MLR$, a set of measure $1$.  Hence $1-\lambda([S]) = \sum_{i \geq 2}\lambda([S_i])$.  For $i > 0$, the fact that $[S_i] \subseteq \MP{U}_{i-1}$ implies that $\lambda([S_i]) \leq \lambda(\MP{U}_{i-1}) \leq 2^{-(i-1)}$.  Lemma~\ref{lem-SumRightCE} thus applies to the sequence $(\lambda([S_i]))_{i \geq 2}$, so $1-\lambda([S])$ is indeed right-c.e.
\end{proof}

\begin{Corollary}
For every universal ML-test $\MP U$, there is a function that is exactly $\MP U$-layerwise computable but not $\MP U$-layerwise computable.  Thus the  class of exactly layerwise computable functions is always strictly bigger than the class of layerwise computable functions.\qed
\end{Corollary}

We now prove the analog of Corollary~\ref{cor-LWCDependsOnTest} for exact layerwise computability:  the notion of exact layerwise computability also depends on the universal ML-test used to define it.

\begin{Definition}
Let $\MP U$ be an ML-test.  Define the \emph{stratification} of $\MP U$ to be the ML-test $\MP{U}^\str$ given by
\begin{align*}
\MP{U}^\str_i = [1^{i+3}] \cup \{1^\ell0^\smf X : \ell \in \omega \andd X \in \MP{U}_{i+1}\}
\end{align*}
for each $i \in \omega$.
\end{Definition}

$\MP{U}^\str_i$ can be enumerated by enumerating $[1^{i+3}]$, then by initiating the enumeration of the cones $[1^\ell0^\smf \sigma]$ whenever $[\sigma]$ is enumerated as a subset of $\MP{U}_{i+1}$.  One readily checks that $\lambda(\MP{U}^\str_i) \leq 2^{-(i+3)} + \lambda(\MP{U}_{i+1}) \leq 2^{-(i+3)} + 2^{-(i+1)} \leq 2^{-i}$.  Thus $\MP{U}^\str$ is an ML-test.

\begin{Lemma}\label{lem-StrUniversal}
If $\MP U$ is a universal ML-test, then $\MP{U}^\str$ is a universal ML-test.
\end{Lemma}

\begin{proof}
Let $\MP U$ be a universal ML-test.  Let $X \in \overline\MLR$.  We need to show that $X \in \bigcap_{i \in \omega}\MP{U}^\str_i$.  If $X$ is the sequence of all $1$'s, then clearly $X \in \bigcap_{i \in \omega}\MP{U}^\str_i$.  Otherwise, $X = 1^\ell0^\smf Y$ for some $\ell \in \omega$ and some $Y \in \overline\MLR$.  As $Y \in \overline\MLR$, for every $i \in \omega$ there is a $[\tau_i] \subseteq \MP{U}_{i+1}$ with $Y \in [\tau_i]$.  Then, by the definition of the stratification, $X \in [1^\ell0^\smf\tau_i] \subseteq \MP{U}^\str_i$.  Thus $X \in \bigcap_{i \in \omega}\MP{U}^\str_i$.
\end{proof}

\begin{Lemma}\label{lem-StrRD}
Let $\MP U$ be a universal ML-test, and let $X \in \MLR$ be such that $\rd_{\MP U}(X) > 0$.  Then $\forall \ell(\rd_{\MP{U}^\str}(1^\ell0^\smf X) = \rd_{\MP U}(X) - 1)$.
\end{Lemma}

\begin{proof}
Let $d = \rd_{\MP U}(X)$.  Then $X \in \MP{U}_i$ for all $i < d$, and $X \notin \MP{U}_d$.  Thus $1^\ell0^\smf X \in \MP{U}^\str_i$ for all $i < d-1$, and $1^\ell0^\smf X \notin \MP{U}^\str_{d-1}$.  Hence $\rd_{\MP{U}^\str}(1^\ell0^\smf X) = \rd_{\MP U}(X) - 1$.
\end{proof}

\begin{Theorem}
The notion of exact layerwise computability depends on the universal ML-test.  That is, there are universal ML-tests $\MP U$ and $\MP W$ and a function $F$ that is exactly $\MP U$-layerwise computable but not exactly $\MP W$-layerwise computable.
\end{Theorem}

\begin{proof}
Let $\MP V$ be any universal ML-test satisfying $\forall i(\lambda(\MP{V}_i) \leq 2^{-(i+2)})$, and notice that $\MP{V}^\str$ also satisfies $\forall i(\lambda(\MP{V}^\str_i) \leq 2^{-(i+2)})$.  Define an ML-test $\MP U$ by $\MP{U}_0 = \MP{V}^\str_0$ and
\begin{align*}
\MP{U}_{i+1} = \MP{V}^\str_{i+1} \cup \{1^e0^\smf X : \Phi_e(e)\da \andd 1^e0^\smf X \in \MP{V}^\str_i\}
\end{align*}
for each $i \in \omega$.  The $\MP{U}_{i+1}$ can be enumerated by enumerating $\MP{V}^\str_{i+1}$ while watching the computations of the $\Phi_e(e)$.  When a $\Phi_e(e)$ converges, $\MP{U}_{i+1}$ begins enumerating $[1^e0] \cap \MP{V}^\str_i$ as well.  For each $i \in \omega$, $\lambda(\MP{U}_{i+1}) \leq \lambda(\MP{V}^\str_{i+1}) + \lambda(\MP{V}^\str_i) \leq 2^{-(i+3)} + 2^{-(i+2)} \leq 2^{-(i+1)}$.  Thus $\MP U$ is a ML-test.  $\MP U$ is universal because $\MP{V}^\str$ is universal by Lemma~\ref{lem-StrUniversal}.

Trivially, $\rd_{\MP U}$ is exactly $\MP U$-layerwise computable.  We show that $\rd_{\MP U}$ is not exactly $\MP{V}^\str$-layerwise computable.  Suppose for a contradiction that $\Phi$ witness that $\rd_{\MP U}$ is exactly $\MP{V}^\str$-layerwise computable.  Let $X$ be ML-random and such that $\rd_{\MP V}(X) \geq 2$.  We show that $X \geqT 0'$, which is a contradiction because there are such $X$'s that do not compute $0'$.  Let $d = \rd_{\MP V}(X)$.  Given $e$, compute $d' = \Phi(1^e0^\smf X)(d-1)$.  Output $0$ if $d' = d-1$, and output $1$ otherwise.  By Lemma~\ref{lem-StrRD}, $d-1 = \rd_{\MP{V}^\str}(1^e0^\smf X)$, so by assumption $\Phi(1^e0^\smf X)(d-1)$ must produce $d' = \rd_{\MP U}(1^e0^\smf X)$.  If $\Phi_e(e)\ua$, then $\forall i ([1^e0] \cap \MP{U}_i = [1^e0] \cap \MP{V}^\str_i)$, so $d' = \rd_{\MP U}(1^e0^\smf X) = \rd_{\MP{V}^\str}(1^e0^\smf X) = d-1$.  Thus the procedure correctly outputs $0$.  Now suppose $\Phi_e(e)\da$.  Then $(\forall i < d-1)(1^e0^\smf X \in \MP{V}^\str_i)$, so $(\forall i \leq d-1)(1^e0^\smf X \in \MP{U}_i)$.  Thus $d' = \rd_{\MP U}(1^e0^\smf X) > d-1$, so the procedure correctly outputs $1$.  Thus $X \geqT 0'$, which gives the desired contradiction.

Let $\MP W = \MP{V}^\str$, and let $F = \rd_{\MP U}$.  Then $F$ is exactly $\MP U$-layerwise computable but not exactly $\MP W$-layerwise computable.
\end{proof}

\section{Randomness deficiency and the Weihrauch degrees}\label{sec-Weihrauch}

In this section we study randomness deficiency in the context of the Weihrauch degrees.  Our results show that, unlike the situation with layerwise computability, randomness deficiency considerations in the Weihrauch degrees are typically not sensitive to the choice of the universal ML-test.

\begin{Definition}\label{def-LAY}
Let $\MP U$ be a universal ML-test.
\begin{itemize}
\item $\LAY_{\MP U}$ is the multi-valued function $\MLR \rightrightarrows \omega$ defined by $\LAY_{\MP U}(X) = \{i : X \notin \MP{U}_i\}$.

\item $\RD_{\MP U}$ is the single-valued function $\MLR \rightarrow \omega$ defined by $\RD_{\MP U}(X) = \rd_{\MP U}(X)$.
\end{itemize}
\end{Definition}

We intend the natural representations of $\MLR$ and $\omega$.  $\MLR$ is represented by the identity function on the domain $\MLR$ (i.e., the characteristic functions of ML-random sets).  For $\omega$, we use the identification of $n$ and $\{n\}$.  Thus $\omega$ is represented by the function $\{n\} \mapsto n$ whose domain is the set of characteristic functions of singleton subsets of $\omega$.  In full technicality, the realizers of $\LAY_{\MP U}$ are exactly the functions $\Gamma \colon \MLR \imp 2^\omega$ such that, for all $X \in \MLR$, $\Gamma(X) = \chi_{\{i\}}$ for some $i$ with $X \notin \MP{U}_i$.  For convenience, we often think of such a realizer $\Gamma$ more simply as a function $\MLR \imp \omega$ such that $(\forall X \in \MLR)(X \notin \MP{U}_{\Gamma(X)})$.  Similarly, we think of a realizer for $\RD_{\MP U}$ as the function $\rd_{\MP U}$.  Thus, in a sense, $\RD_{\MP U}$ is the same object as $\rd_{\MP U}$, but for the sake of conceptual clarity we use the notation `$\rd_{\MP U}$' to refer to the $\MP U$-randomness deficiencies of ML-random sets, and we use the notation `$\RD_{\MP U}$' to refer to the mathematical task of determining the $\MP U$-randomness deficiencies of ML-random sets.

In contrast to Corollary~\ref{cor-LWCDependsOnTest}, which states that the notion of layerwise computability depends on the universal ML-test used to define it, the strong Weihrauch degree of $\LAY_{\MP U}$ is independent of the universal ML-test $\MP U$.  Therefore we are henceforth justified in writing $\LAY$ in place of $\LAY_{\MP U}$ in the context of both the Weihrauch degrees and the strong Weihrauch degrees.

\begin{Proposition}\label{prop-LAYisLAY}
If $\MP U$ and $\MP V$ are universal ML-tests, then $\LAY_{\MP V} \leqsW \LAY_{\MP U}$.
\end{Proposition}

\begin{proof}
We describe a Turing functional $\Phi$ such that $\Phi(\MLR) \subseteq \MLR$ and
\begin{align*}
(\forall X \in \MLR)(\forall i)(i \in \LAY_{\MP U}(\Phi(X)) \rightarrow i \in \LAY_{\MP V}(X)).
\end{align*}
Hence, by letting $\Psi$ be the identity function, we have that $\LAY_{\MP V} \leqsW \LAY_{\MP U}$.

First define an auxiliary universal ML-test $\MP{V}'$ by $\MP{V}_i' = \bigcup_{j > i}\MP{V}_i$ for each $i \in \omega$.  It is easy to see that $\MP{V}'$ is indeed a universal ML-test.  Furthermore, $\MP{V}'$ has the property that if $X \in \MLR$ and $i > \rd_{\MP{V}'}(X)$, then $i \in \LAY_{\MP V}(X)$ (simply because by definition $X \notin \MP{V}_{\rd_{\MP{V}'(X)}}' = \bigcup_{i > \rd_{\MP{V}'}(X)}\MP{V}_n$).

Given $X \in 2^\omega$, to compute $\Phi(X)$, output the bits of $X$ while searching for a stage $s_0$ such that $X \in \MP{V}'_{0,s_0}$.  If $s_0$ is found, let $\sigma_0$ be the output of $\Phi(X)$ so far, and search for a string $\tau_0$ such that $[{\sigma_0}^\smf\tau_0] \subseteq \MP{U}_0 \cap \MP{U}_1$.  Such a $\tau_0$ exists because there are non-ML-random sequences extending $\sigma_0$, and, as a universal test, $\MP U$ must capture all of these.  Append $\tau_0$ to the current output of $\Phi(X)$ (so that it becomes ${\sigma_0}^\smf\tau_0$), and restart outputting the bits of $X$ from the beginning.  While $\Phi(X)$ is outputting the bits of ${\sigma_0}^\smf{\tau_0}^\smf X$, search for a stage $s_1$ such that $X \in \MP{V}'_{1,s_1}$.  If $s_1$ is found, let $\sigma_1$ be the output of $\Phi(X)$ so far, and search for a $\tau_1$ such that $[{\sigma_1}^\smf\tau_1] \subseteq \MP{U}_0 \cap \MP{U}_1 \cap \MP{U}_2$, which again exists because there are non-ML-random sequences extending $\sigma_1$.  Append $\tau_1$ to the current output of $\Phi(X)$ (so that it becomes ${\sigma_1}^\smf\tau_1$), and restart outputting the bits of $X$ from the beginning.  Continue in this way, now searching for $s_2$ such that $X \in \MP{V}'_{2,s_2}$ and so on.  This defines $\Phi$.

Now consider an $X \in \MLR$, and let $v = \rd_{\MP{V}'}(X)$.  As $X \notin \MP{V}'_v$, the output of $\Phi(X)$ is a sequence of the form $\sigma^\smf X$, which is in $\MLR$ because $X$ is in $\MLR$.  Moreover, the construction clearly ensures that if $i \leq v$, then $\Phi(X) \in \MP{U}_i$.  Hence $\rd_{\MP U}(\Phi(X)) > \rd_{\MP{V}'}(X)$.  Therefore, if $i \in \LAY_{\MP U}(\Phi(X))$, then $i \geq \rd_{\MP U}(\Phi(X)) > \rd_{\MP{V}'}(X)$, so $i \in \LAY_{\MP V}(X)$ as desired.
\end{proof}

\begin{Remark}\label{rmk-LAYisLAY}
In the last line of the proof of Proposition~\ref{prop-LAYisLAY}, the fact that $i \in \LAY_{\MP U}(\Phi(X))$ only matters insofar as it implies that $i \geq \rd_{\MP U}(\Phi(X))$.  For a universal ML-test $\MP U$, let $\LAY^\alt_{\MP U}$ be the multi-valued function $\MLR \rightrightarrows \omega$ defined by $\LAY^\alt_{\MP U}(X) = \{i : i \geq \rd_{\MP U}(X)\}$.  Then for any universal ML-tests $\MP U$ and $\MP V$, we have that $\LAY_{\MP V} \leqsW \LAY^\alt_{\MP U}$ because the proof of  Proposition~\ref{prop-LAYisLAY} goes through with $\LAY^\alt_{\MP U}$ in place of $\LAY_{\MP U}$.  As clearly $\LAY^\alt_{\MP U} \leqsW \LAY_{\MP U}$, we have that $\LAY_{\MP U} \equivsW \LAY_{\MP V} \equivsW \LAY^\alt_{\MP U} \equivsW \LAY^\alt_{\MP V}$ for every pair $\MP U$ and $\MP V$ of universal ML-tests.  Plainly stated, given a universal ML-test $\MP U$ and an $X \in \MLR$, determining an $i$ such that $X \notin \MP{U}_i$ is the same as determining an upper bound for $\rd_{\MP U}(X)$, up to strong Weihrauch equivalence.
\end{Remark}

Furthermore, $\LAY_{\MP U}$ and $\RD_{\MP V}$ are Weihrauch-equivalent for any two universal ML-tests.  Thus in the context of the Weihrauch degrees we may unambiguously write either $\LAY$ or $\RD$ for any function $\LAY_{\MP U}$ or $\RD_{\MP U}$ for any universal ML-test $\MP U$.

\begin{Theorem}\label{thm-LAYisRD}
Let $\MP U$ and $\MP V$ be universal ML-tests.  Then $\RD_{\MP V} \equivW \LAY_{\MP U}$.
\end{Theorem}
\begin{proof}
First, 
%$\LAY_{\MP U} \leqW \RD_{\MP V}$ follows from  because 
we trivially have ${\LAY_{\MP V} \leqsW \RD_{\MP V}}$, and by Proposition~\ref{prop-LAYisLAY} therefore $$\LAY_{\MP U} \leqsW \LAY_{\MP V} \leqsW \RD_{\MP V}.$$  
%So in this direction we have even that $\LAY_{\MP U} \leqsW \RD_{\MP V}$.

For the other direction we prove $\RD_{\MP V} \leqW \LAY_{\MP U}$. The proof is very similar to that of Proposition~\ref{prop-LAYisLAY}. Compute $\Phi$ and $\Psi$ as follows.  $\Phi(X)$ outputs the bits of $X$ while searching for a stage $s$ such that $X \in \MP{V}_{0,s}$.  If $s$ is found, let $\sigma$ be the output of $\Phi(X)$ so far, and search for a string $\tau$ such that $[\sigma^\smf\tau] \subseteq \bigcap_{i \leq s}\MP{U}_i$.  Append $\tau$ to the current output of $\Phi(X)$ (so that it becomes $\sigma^\smf\tau$), and restart outputting the bits of $X$ from the beginning.  Now search for a stage $s$ such that $X \in \MP{V}_{1,s}$.  If $s$ is found, let $\sigma$ be the output of $\Phi(X)$ so far, and search for a $\tau$ such that $[\sigma^\smf\tau] \subseteq \bigcap_{i \leq s}\MP{U}_i$.  Append $\tau$ to the current output of $\Phi(X)$ (so that it becomes $\sigma^\smf\tau$), and restart outputting the bits of $X$ from the beginning.  Continue in this way, now searching for an $s$ such that $X \in \MP{V}_{2,s}$ and so on.  This defines $\Phi$.  $\Psi(\la X, k \ra)$ outputs the least $i$ such that $X \notin \MP{V}_{i,k}$.

To see that $\Phi$ and $\Psi$ witness that $\RD_{\MP V} \leqW \LAY_{\MP U}$, consider an $X \in \MLR$, and let $v = \rd_{\MP V}(X)$.  As $X \notin \MP{V}_v$, the output of $\Phi(X)$ is a sequence of the form $\sigma^\smf X$, which is in $\MLR$ because $X$ is in $\MLR$.  Hence $\Phi(X)$ is in the domain of $\LAY_{\MP U}$.  Let $u = \rd_{\MP U}(\Phi(X))$.  The construction ensures that $u$ is large enough so that if $i < v$, then $X \in \MP{V}_{i,u}$.  Thus if $k \in \LAY_{\MP U}(\Phi(X))$, then $k \geq u$, so $\Psi(\la X, k \ra)$, the least $i$ such that $X \notin \MP{V}_{i,k}$, is exactly $v$ as desired.
\end{proof}

However, Theorem~\ref{thm-LAYisRD} cannot be improved to have $\equivsW$ in place of $\equivW$.

\begin{Proposition}
Let $\MP U$ and $\MP V$ be universal ML-tests.  Then $\RD_{\MP V} \nleqsW \LAY_{\MP U}$.
\end{Proposition}

\begin{proof}
In light of Proposition~\ref{prop-LAYisLAY}, we may assume that $\MP U$ is nested.  Suppose for a contradiction that $\RD_{\MP V} \leqsW \LAY_{\MP U}$, and let $\Phi$ and $\Psi$ witness the reduction.  That is, if $X \in \MLR$, then $\Phi(X) \in \MLR$, and if $n \geq \rd_{\MP U}(\Phi(X))$, then $\Psi(n) = \rd_{\MP V}(X)$.  Now let $X, Y \in \MLR$ be such that $\rd_{\MP V}(X) \neq \rd_{\MP V}(Y)$.  Let $n \geq \max(\rd_{\MP U}(\Phi(X)), \rd_{\MP U}(\Phi(Y)))$.  Then $\rd_{\MP V}(X) = \Psi(n) = \rd_{\MP V}(Y)$, a contradiction.
\end{proof}

We were unable to determine whether or not $\RD_{\MP U} \equivsW \RD_{\MP V}$ for every pair $\MP U$ and $\MP V$ of universal ML-tests, so we leave it as a question.

\begin{Question}
Let $\MP U$ and $\MP V$ be universal ML-tests.  Does $\RD_{\MP U} \equivsW \RD_{\MP V}$ hold?
\end{Question}

Now we show that $\LAY$ enjoys various idempotence properties in the Weihrauch degrees and in the strong Weihrauch degrees.  These properties are ultimately due to the fact that, given a universal ML-test $\MP U$ and two ML-random sequences $X$ and $Y$, we can effectively generate a ML-random sequence $Z$ such that $\rd(Z) \geq \max\{\rd(X), \rd(Y)\}$.

\begin{Proposition}\label{prop-LAYxLAY}
$\LAY \times \LAY \leqsW \LAY$.
\end{Proposition}

\begin{proof}[Proof sketch]
Fix a nested universal ML-test $\MP U$.  Given $X, Y \in 2^\omega$, compute $\Phi(\la X, Y \ra)$ in the style of Proposition~\ref{prop-LAYisLAY} by copying the bits of $X$ and increasing $\rd(\Phi(\la X, Y \ra))$ whenever increases in $\rd(X)$ or $\rd(Y)$ are noticed.  Let $\Psi$ be $n \mapsto \la n, n \ra$.  If $X$ and $Y$ are ML-random, then $\Phi(X,Y) = \sigma^\smf X$ for some $\sigma$ for which $\rd(\sigma^\smf X) \geq \max(\rd(X), \rd(Y))$.
\end{proof}

It follows that $\LAY \times \LAY \equivsW \LAY$.
% and that $\LAY \times \LAY \equivW \LAY$.

$\LAY$ is also strongly Weihrauch-equivalent to its parallelization when restricting to nested universal ML-tests and sequences of ML-random sets of uniformly bounded randomness deficiency.  Let $\MP U$ be a universal ML-test.  Let $\widetilde{\LAY_{\MP U}}$ be the restriction of $\widehat{\LAY_{\MP U}}$ to the domain consisting of all (codes for) sequences $(X_i)_{i \in \omega}$ in $\MLR^\omega$ such that $(\exists n)(\forall m \geq n)(\forall i)(X_i \notin \MP{U}_m)$.  Note that if $\MP U$ is nested, this condition is equivalent to $\exists n \forall i (\rd_{\MP U}(X_i) \leq n)$.

\begin{Proposition}
Let $\MP U$ be a universal ML-test.  Then $\widetilde{\LAY_{\MP U}} \leqsW \LAY$.
\end{Proposition}
\begin{proof}[Proof sketch]

Given $\vv{X} = (X_i)_{i \in \omega}$ in $\MLR^\omega$ such that $(\exists n)(\forall m \geq n)(\forall i)(X_i \notin \MP{U}_m)$, we need only have $\Phi(\vv X)$ produce an ML-random set $Y$ such that $\rd(Y) \geq n$, for then we can let $\Psi$ be the functional that maps $n$ to the constant sequence of $\omega$ copies of $n$.

Compute $\Phi(\vv X)$ in the style of Proposition~\ref{prop-LAYisLAY} by copying the bits of $X_0$ and watching the sets $X_i$ enter the components of $\MP U$ in a dovetailing fashion.  If at stage $s = \la i, n, t \ra$ we notice that $X_i \in \MP{U}_{n,t}$ but that $[\sigma_s] \nsubseteq \bigcap_{m \leq n} \MP{U}_{m,s}$ (where $\sigma_s$ is the string output by $\Phi(\vv X)$ so far), then extend $\sigma_s$ to a ${\sigma_s}^\smf \tau$ such that $[{\sigma_s}^\smf \tau] \subseteq \bigcap_{m \leq n}\MP{U}_m$, and restart outputting the bits of $X_0$ from the beginning.  If indeed $\vv X$ satisfies $(\exists n)(\forall m \geq n)(\forall i)(X_i \notin \MP{U}_m)$, then by some stage $\Phi(\vv X)$ will have produced an output $\sigma$ such that $[\sigma] \subseteq \bigcap_{m < n}\MP{U}_m$ for a witnessing $n$.  Thus in the end $\Phi(\vv X) = \sigma^\smf X_0$, where $\sigma^\smf X_0 \in \MLR$ with $\rd(\sigma^\smf X_0) \geq n$.
\end{proof}

\begin{Theorem}\label{thm-LAYstarLAYisLAY}
$\LAY \ast \LAY \leqW \LAY$.
\end{Theorem}

\begin{proof}
In fact we prove that $\LAY \ast \LAY \leqW \RD \times \RD$, from which the theorem follows because $\RD \times \RD \equivW \LAY \times \LAY \equivW \LAY$ by Theorem~\ref{thm-LAYisRD} and Proposition~\ref{prop-LAYxLAY}.

Fix any nested universal ML-test $\MP U$.  Let $f \colon (R, \delta_R) \rightrightarrows (S, \delta_S)$ and $g \colon (Q, \delta_Q) \rightrightarrows (R, \delta_R)$ be composable functions on represented spaces such that $f,g \leqW \LAY$.  We show that $f \circ g \leqW \RD \times \RD$.  To this end, let $\Phi_f$, $\Psi_f$, $\Phi_g$, and $\Psi_g$ be Turing functionals such that $\Psi_f \circ \la \id, \Gamma \circ \Phi_f \ra \vdash f$ and $\Psi_g \circ \la \id, \Gamma \circ \Phi_g \ra \vdash g$ whenever $\Gamma \vdash \LAY$.  What we need are Turing functionals $\Phi$ and $\Psi$ such that $\Psi \circ \la \id, \Delta \circ \Phi \ra \vdash f \circ g$ whenever $\Delta \vdash \RD \times \RD$.

Consider a realizer $\Gamma \vdash \LAY$ and an $X \in \dom(g \circ \delta_Q)$.  Then $Y = \Phi_g(X)$ must be ML-random because it must be in $\dom(\Gamma)$.  Likewise, $\Psi_g(\la X, \rd(Y) \ra) \in \dom(f \circ \delta_R)$, so ${Z = \Phi_f(\Psi_g(\la X, \rd(Y) \ra))}$ must be ML-random because it must be in $\dom(\Gamma)$.  

The task of the functional $\Phi$ is, given $X \in \dom(g \circ \delta_Q)$, to produce a pair $\Phi(X) = \la Y, W \ra$, where $Y = \Phi_g(X)$ and $\rd(W) \geq \rd(Z)$.  $\Phi(X)$ computes $Y$ simply by running $\Phi_g(X)$.  $\Phi(X)$ computes $W$ by watching $\rd(Y)$ increase, simulating the computation of $Z$, and watching $\rd(Z)$ increase. Specifically, to compute $W$ given $X$, begin by setting $d_Y = d_Z = 0$, and output the bits of $\Phi_g(X)$ while searching for a stage at which either $\Phi_g(X) \in \MP{U}_{d_Y}$ or $\Phi_f(\Psi_g(\la X, d_Y \ra)) \in \MP{U}_{d_Z}$.  If $\Phi_g(X) \in \MP{U}_{d_Y}$ is witnessed, update $d_Y = d_Y + 1$.  If $\Phi_f(\Psi_g(\la X, d_Y \ra)) \in \MP{U}_{d_Z}$ is witnessed, let $\sigma$ be the initial segment of $W$ that has been produced thus far, and search for a string $\tau$ such that $[\sigma^\smf\tau] \subseteq \MP{U}_{d_Z}$.  Append $\tau$ to the output thus far (so that it becomes $\sigma^\smf\tau$), update $d_Z = d_Z+1$, and restart outputting the bits of $\Phi_g(X)$ from the beginning.  $\Psi$ is the functional $\la X, \la n, m \ra \ra \mapsto \Psi_f(\la \Psi_g(\la X, n \ra), m \ra)$.

Let $\Delta \vdash \RD \times \RD$, and let $X \in \dom(g \circ \delta_Q)$.  We need to show that
\begin{align*}
\delta_S(\Psi(\la X, \Delta(\Phi(X)) \ra)) \in f(g(\delta_Q(X))).
\end{align*}
As discussed above, $Y = \Phi_g(X)$ is ML-random, so as the computation of $\Phi(X)$ proceeds, the variable $d_Y$ increases until it reaches its final value of $n = \rd(Y)$.  $Z = \Phi_f(\Psi_g(\la X, n \ra))$ is also ML-random, so once $d_Y$ reaches $\rd(Y)$, either $d_Z \geq \rd(Z)$ already, in which case $d_Z$ never changes again, or $d_Z < \rd(Z)$, in which case $d_Z$ increases until it reaches its final value of $\rd(Z)$.  Thus $\Phi(X) = \la Y, W \ra$, where $Y = \Phi_g(X)$ and $W = \sigma^\smf Y$ for some string $\sigma$.  $W$ is ML-random because $Y$ is ML-random, and the construction ensures that $\sigma$ satisfies $[\sigma] \subseteq \MP{U}_i$ for all $i < \rd(Z)$, which implies that $\rd(W) \geq \rd(Z)$.

As $\Phi(X) = \la Y, W \ra \in \dom(\Delta)$ and $\Delta \vdash \RD \times \RD$, we have that $\Delta(\Phi(X)) = \la n, m \ra$, where $n = \rd(Y)$ and $m = \rd(W) \geq \rd(Z)$.  Thus $\delta_R(\Psi_g(\la X, n \ra)) \in g(\delta_Q(X)) \subseteq \dom(f)$.  Thus $\Psi_g(\la X, n \ra) \in \dom(f \circ \delta_R)$.  Thus
\begin{align*}
\delta_S(\Psi_f(\la \Psi_g(\la X, n \ra), m \ra)) \in f(\delta_R(\Psi_g(\la X, n \ra))) \subseteq f(g(\delta_Q(X))).  
\end{align*}
As
\begin{align*}
\Psi(\la X, \Delta(\Phi(X)) \ra) = \Psi(\la X, \la n, m \ra \ra) = \Psi_f(\la \Psi_g(\la X, n \ra), m \ra),
\end{align*}
we have the desired $\delta_S(\Psi(\la X, \Delta(\Phi(X)) \ra)) \in f(g(\delta_Q(X)))$.
\end{proof}

\begin{Remark}
At this point it is natural to ask whether the strong Weihrauch analog of Theorem~\ref{thm-LAYstarLAYisLAY} holds.  Unfortunately, the strong Weihrauch degree $\LAY \ast_{\mathrm{sW}} \LAY$ does not exist.  In fact, no two functions $f, g \leqsW \LAY$ can be composed.  Consider first a function $f \colon (R, \delta_R) \rightrightarrows (S, \delta_S)$ that is $\leqsW \LAY$.  We show that no $Y \in \dom(f \circ \delta_R)$ is computable.  Let $\Phi$ and $\Psi$ witness that $f \leqsW \LAY$, and let $\Gamma \vdash \LAY$.  If $Y \in \dom(f \circ \delta_R)$, then $\delta_S(\Psi(\Gamma(\Phi(Y))))$ must be defined.  This means that $\Gamma(\Phi(Y))$ must be defined, so $\Phi(Y)$ must be in $\MLR$.  Thus $Y$ cannot be computable because if it were then $\Phi(Y)$ would be computable as well, which would contradict its ML-randomness.  Now consider a function $g \colon (Q, \delta_Q) \rightrightarrows (R, \delta_R)$ that is $\leqsW \LAY$, let $\Phi$ and $\Psi$ witness that $g \leqsW \LAY$, and let $\Gamma \vdash \LAY$.  Let $X \in \dom(g \circ \delta_Q)$, and let $Y = \Psi(\Gamma(\Phi(X)))$.  Then $\delta_R(Y) \in g(\delta_Q(X))$, so if $f$ and $g$ were composable, we would have that $\delta_R(Y) \in \dom(f)$, which we can rewrite as $Y \in \dom(f \circ \delta_R)$.  However, $\Gamma(\Phi(X))$ is just (the characteristic function of) a number, so $Y$ is computable.  This is a contradiction.
\end{Remark}

Closed choice principles (see Brattka and Gherardi~\cite{MR2760117}) have become important benchmarks in the Weihrauch degrees.  We now relate $\LAY$ to $\C_\N$, the closed choice principle on $\omega$ (the subscript on `$\C_\N$' is `$\N$' instead of `$\omega$' just to match the notation established in previous studies).

\begin{Definition}
$\C_\N \colon \!\!\!\subseteq\! \omega^\omega \rightrightarrows \omega$ is the multi-valued function whose domain is $\{f \in \omega^\omega : \exists n \forall k(f(k) \neq n+1)\}$ and is defined by $\C_\N(f) = \omega \setminus \{n : \exists k(f(k) = n+1)\}$ for every $f \in \dom(\C_\N)$.
\end{Definition}

The intuition behind the definition of $\C_\N$ is that an $f \in \dom(\C_\N)$ codes the complement of a non-empty subset of $\omega$ and that, given such an $f$, $\C_\N(f)$ must produce a member of the set that $f$ codes.  (The reason for `$n+1$' instead of `$n$' in the definition is to allow the constantly $0$ function to code $\omega$.)  We show that although $\C_\N$ is strictly stronger than $\LAY$, the two degrees become Weihrauch-equivalent when the inputs of $\C_\N$ are attached to ML-random sets.
 
\begin{Proposition}\label{prop-LAYvsC}
$\LAY \leqsW \C_\N$ and $\C_\N \nleqW \LAY$.
\end{Proposition}
\begin{proof}
Fix a universal ML-test $\MP U$.  For $\RD \leqsW \C_\N$, let $p_i$ denote the $i$\textsuperscript{th} prime for each $i \in \omega$.  Compute $\Phi$ and $\Psi$ as follows.  In the computation of $\Phi(X)$, `enumerate $n$' means set $\Phi(X)(k) = n+1$, where $k$ is least such that $\Phi(X)(k)$ has not yet been defined.  (The discrepancy between $n$ and $n+1$ here reflects the discrepancy between $n$ and $n+1$ in the definition of $\C_\N$.)

Given $X \in 2^\omega$, $\Phi(X)$ begins by enumerating all numbers that are not $p_0$ while searching for a stage at which $X$ enters $\MP{U}_0$.  When such a stage is found, $\Phi(X)$ finds the least $n$ such that $p_1^n$ is greater than all numbers it has enumerated so far, stops its current enumeration, and begins enumerating all numbers that are not $p_1^n$ while searching for a stage at which $X$ enters $\MP{U}_1$.  When such a stage is found, $\Phi(X)$ continues this pattern by finding the least $n$ such that $p_2^n$ is greater than all numbers it has enumerated so far, stopping its current enumeration, and beginning an enumeration of all numbers that are not $p_2^n$ while searching for a stage at which $X$ enters $\MP{U}_2$, and so on.  $\Psi(n)$ is the least $i$ such that $p_i$ divides $n$.

If $X \in \MLR$, then the $\C_\N$-instance coded by $\Phi(X)$ is a singleton of the form $\{p_i^n\}$, where $i = \rd(X)$.  Hence $\Phi$ and $\Psi$ witness that $\RD \leqsW \C_\N$.

To see that $\C_\N \nleqW \LAY$, observe that there are computable elements $f \in \dom(\C_\N)$, and, for such an $f$, if $\Phi$ if is a Turing functional and $\Phi(f)$ is defined, then $\Phi(f)$ is computable and hence not in $\dom(\Gamma) = \MLR$ for any realizer $\Gamma$ of $\LAY$.
\end{proof}

Thus in the strong Weihrauch degrees we have that $\LAY \lesW \RD \lesW \C_\N$, and in the Weihrauch degrees we have that $\LAY \equivW \RD \leW \C_\N$.  We note that the proof of Proposition~\ref{prop-LAYvsC} in fact shows that $\LAY \leqsW \UC_\N$, where $\UC_\N$ is the \emph{unique} choice principle on $\omega$ as defined by Brattka, de Brecht, and Pauly~\cite{MR2915694}.  

The reason that $\C_\N \leqW \LAY$ fails is that $\C_\N$ has instances that cannot compute ML-random sets. One may consider this an unsatisfactory answer and wonder whether in the presence of randomness the separation between both principles goes away. The first result in this direction is by Brattka, Gherardi and Hölzl~\cite{BrHoeGh13} who show that $\LAY \ast \MLR \equivW \C_\N \ast \MLR$.
In fact, it is sufficient to ``tag'' every input to $\C_\N$ with an additional ML-random set to obtain Weihrauch equivalence, as we can see in the following proposition. Here, $\id_\MLR$ is the identity function on the domain $\MLR$.

\begin{Proposition}\label{prop-LAYisCtimesMLR}
$\LAY \equivW \C_\N \times \id_\MLR$.
\end{Proposition}
\begin{proof}[Proof sketch]
We see that $\LAY \leqW \C_\N \times \id_\MLR$ by using the reduction from Proposition~\ref{prop-LAYvsC} and by passing along the given $\LAY$ instance as the input to the $\id_\MLR$ side of $\C_\N \times \id_\MLR$.

For $\C_\N \times \id_\MLR \leqW \LAY$, fix a universal ML-test $\MP U$.  For any function $f \colon \omega \imp \omega$ and any $s \in \omega$, let $a_s(f)$ be the least number $n$ such that $(\forall k \leq s)(f(k) \neq n+1)$.  $\Phi(\la f, X \ra)$ copies the bits of $X$ while checking for each $s$ whether $a_s(f) = a_{s+1}(f)$.  At stages $s$ for which $a_s(f) = a_{s+1}(f)$, $\Phi(\la f, X \ra)$ appends a string to its output to ensure that $\rd(\Phi(\la f, X \ra)) \geq s+1$.  Then $\Phi(\la f, X \ra)$ restarts outputting the bits of $X$ from the beginning.  $\Psi(\la \la f, X \ra, s \ra)$ is $\la a_s(f), X \ra$.  If $\la f, X \ra$ is a $\C_\N \times \id_\MLR$-instance, then $\Phi(\la f, X \ra)$ is a ML-random sequence of the form $\sigma^\smf X$ where, for any $s \geq \rd(\sigma^\smf X)$, $a_s(f)$ is the least $n$ such that $n+1 \notin \ran(f)$.  Thus if $s \geq \rd(\sigma^\smf X)$, then $\Psi(\la \la f, X \ra, s \ra) = \la a_s(f), X \ra \in (\C_\N \times \id_\MLR)(\la f, X \ra)$.
\end{proof}

We thank Davie, Fouch\'e, and Pauly~\cite{PaulyPC} for pointing us to the correct statement of Proposition~\ref{prop-LAYisCtimesMLR}.  They also independently obtained Proposition~\ref{prop-LAYisCtimesMLR}, and their statement of it is much less cumbersome than our original statement of it was.
%~\cite{PaulyPC}.

The next pair of results identify the complexity of sets whose characteristic functions can be reduced to $\LAY$ when restricted to $\MLR$.  The characteristic function of any $\Delta^0_2$ set restricted to $\MLR$ is $\leqW \LAY$, but there is a $\Sigma^0_2$ set whose characteristic function restricted to $\MLR$ is not $\leqW \LAY$.  This contrasts with the complexity of sets whose characteristic functions are layerwise computable, as we see in Section~\ref{sec-LWCvsLAY}.

\begin{Theorem}\label{thm-Delta02LAY}
If $\MP A \subseteq 2^\omega$ is $\Delta^0_2$, then $\chi_{\MP A} \restriction \MLR \leqW \LAY$.
\end{Theorem}

\begin{proof}
Let $\MP A \subseteq 2^\omega$ be $\Delta^0_2$, and let $(T_i)_{i \in \omega}$ and $(S_i)_{i \in \omega}$ be two uniformly computable sequences of subtrees of $2^{<\omega}$ such that $\MP A = \bigcup_{i \in \omega}[T_i]$ and $2^\omega \setminus \MP A = \bigcup_{i \in \omega}[S_i]$.

Fix a nested universal ML-test $\MP U$.  Given $X \in 2^\omega$, to compute $\Phi(X)$, first search for a string $\tau$ such that $[\tau] \subseteq \MP{U}_0$, and output $\tau$ as the initial bits of $\Phi(X)$.  Now start outputting the bits of $X$ while searching for an $n$ such that $X \restriction n \notin T_0 \cup S_0$.  If such an $n$ is found, let $\sigma$ be the output of $\Phi(X)$ so far, and search for a string $\tau$ such that $[\sigma^\smf\tau] \subseteq \MP{U}_1$.  Append $\tau$ to the current output of $\Phi(X)$ (so that it becomes $\sigma^\smf\tau$), and restart outputting the bits of $X$ from the beginning.  Now search for an $n$ such that $X \restriction n \notin T_1 \cup S_1$.  Again, if such an $n$ is found, let $\sigma$ be the output of $\Phi(X)$ so far, and search for a string $\tau$ such that $[\sigma^\smf\tau] \subseteq \MP{U}_2$.  Append $\tau$ to the current output of $\Phi(X)$ (so that it becomes $\sigma^\smf\tau$), and restart outputting the bits of $X$ from the beginning.  Continue in this way, now searching for an $n$ such that $X \restriction n \notin T_2 \cup S_2$, and so on.  To compute $\Psi(\la X, i \ra)$, search for the least $n$ such that either $(\forall j \leq i)(X \restriction n \notin T_j)$ or $(\forall j \leq i)(X \restriction n \notin S_j)$.  If $(\forall j \leq i)(X \restriction n \notin T_j)$, then output $0$; and otherwise output $1$.

We show that $\Phi$ and $\Psi$ witness that $\chi_{\MP A} \restriction \MLR \leqW \LAY$.  Let $X \in \MLR$.  Note that $X$ is in exactly one of $\bigcup_{i \in \omega}[T_i]$ and $\bigcup_{i \in \omega}[S_i]$.  So let $k$ be least such that $X \in [T_k] \triangle [S_k]$.  Then $\Phi(X) = \sigma^\smf X$ for a $\sigma$ with $[\sigma] \subseteq \MP{U}_k$.  Thus $\Phi(X)$ is ML-random with $\rd(\Phi(X)) \geq k$.  Now let $i \geq \rd(\Phi(X))$ and consider $\Psi(\la X, i \ra)$.  If $X \in \MP A$, then $X \in \bigcup_{j \leq i}[T_j]$ but $X \notin \bigcup_{j \leq i}[S_j]$ (because $i \geq \rd(\Phi(X)) \geq k$).  In this case, $\Psi(\la X, i \ra)$ only finds an $n$ such that $(\forall j \leq i)(X \restriction n \notin S_j)$ and correctly outputs $1$.  Similarly, if $X \notin \MP A$, then $X \in \bigcup_{j \leq i}[S_j]$ but $X \notin \bigcup_{j \leq i}[T_j]$.  In this case, $\Psi(\la X, i \ra)$ only finds an $n$ such that $(\forall j \leq i)(X \restriction n \notin T_j)$ and correctly outputs $0$.
\end{proof}

\begin{Remark}
Theorem~\ref{thm-Delta02LAY} can also be seen indirectly.  First, if $\MP A \subseteq 2^\omega$  is $\Delta^0_2$, then $\chi_{\MP A} \leqW \C_\N$.  This can be seen by a proof similar to that of Theorem~\ref{thm-Delta02LAY}:  let $(T_i)_{i \in \omega}$ and $(S_i)_{i \in \omega}$ be two uniformly computable sequences of subtrees of $2^{<\omega}$ such that $\MP A = \bigcup_{i \in \omega}[T_i]$ and $2^\omega \setminus \MP A = \bigcup_{i \in \omega}[S_i]$.  Given $g \in 2^\omega$, $\Phi(g)$ enumerates the set $\{n+1 : (\forall i \leq n)(g \notin [T_i] \cup [S_i])\}$.  Then, given $n$ such that $n+1 \notin \ran( \Phi(g))$, $\Psi(\la g, n \ra)$ decodes whether or not $g$ is in $\MP A$ as it does in the proof of Theorem~\ref{thm-Delta02LAY}. Notice that this argument also works in Baire space.

 Second, if $\MP A \subseteq 2^\omega$ and $\chi_{\MP A} \leqW \C_\N$, then $\chi_{\MP A}\restriction \MLR \leqW \LAY$.  To see this, observe that $\LAY \equivW \C_\N \times \id_\MLR$ by Proposition~\ref{prop-LAYisCtimesMLR} and that $\chi_{\MP A}\restriction \MLR \leqW \C_\N \times \id_\MLR$ by using the assumed reduction $\chi_{\MP A} \leqW \C_\N$ and the fact that inputs to $\chi_{\MP A}\restriction \MLR$ must be in $\MLR$.  Thus if $\MP A \subseteq 2^\omega$ is $\Delta^0_2$, then $\chi_{\MP A} \leqW \C_\N$, so $\chi_{\MP A}\restriction \MLR \leqW \LAY$. 
\end{Remark}

\begin{Theorem}\label{thm-Sigma2NotLAY}
There is a $\Sigma^0_2$ set $\MP A \subseteq 2^\omega$ such that $\chi_{\MP A} \restriction \MLR \nleqW \LAY$.
\end{Theorem}

\begin{proof}
For the purposes of this proof, let $(\Psi_j)_{j \in \omega}$ be a duplicate effective listing of all Turing functionals.  Fix a universal ML-test $\MP U$.  Define a computable function $k \colon \omega \times 2^{<\omega} \rightarrow \omega$ by letting $k(\sigma,s)$ be the least $i$ such that $[\sigma] \nsubseteq \MP{U}_{i,s}$.  For each $n \in \omega$, let $\tau_n$ be the string $0^n1$.  Let $\MP A$ be the $\Sigma^0_2$ set
\begin{align*}
\MP A = \{{\tau_{\la i,j \ra}}^\smf X : (\exists s_0)(\forall s > s_0)(\Psi_j(\la {\tau_{\la i,j \ra}}^\smf X, k(\Phi_{i,s}({\tau_{\la i,j \ra}}^\smf X),s) \ra) \neq 1)\}.
\end{align*}
We show that $\chi_{\MP A} \restriction \MLR \nleqW \LAY$.  Suppose for a contradiction that $\Phi_i$ and $\Psi_j$ witness that $\chi_{\MP A} \restriction \MLR \leqW \LAY$.  Let $X \in \MLR$.  Then ${\tau_{\la i,j \ra}}^\smf X \in \MLR$, so $\Phi_i({\tau_{\la i,j \ra}}^\smf X) \in \MLR$ as well.  Let $d = \rd(\Phi_i({\tau_{\la i,j \ra}}^\smf X))$, and let $s_0$ be large enough to witness this.  That is, let $s_0$ be such that $(\forall \ell < d)([\Phi_{i,{s_0}}({\tau_{\la i,j \ra}}^\smf X)] \subseteq \MP{U}_{\ell,s_0})$, and observe that $(\forall s > s_0)(k(\Phi_{i,s}({\tau_{\la i,j \ra}}^\smf X),s) = d)$.  Suppose that $\Psi_j(\la {\tau_{\la i,j \ra}}^\smf X, d \ra) = 0$.  Then, for all $s > s_0$, $\Psi_j(\la {\tau_{\la i,j \ra}}^\smf X, k(\Phi_{i,s}({\tau_{\la i,j \ra}}^\smf X),s)\ra) = \Psi_j(\la {\tau_{\la i,j \ra}}^\smf X, d \ra) = 0 \neq 1$.  So, by definition of $\MP A$, ${\tau_{\la i,j \ra}}^\smf X \in \MP A$, contradicting that $\Phi_i$ and $\Psi_j$ witness that $\chi_{\MP A} \restriction \MLR \leqW \LAY$.  Now suppose that $\Psi_j(\la {\tau_{\la i,j \ra}}^\smf X, d \ra) = 1$.  Then, for all $s > s_0$, $\Psi_j(\la {\tau_{\la i,j \ra}}^\smf X, k(\Phi_{i,s}({\tau_{\la i,j \ra}}^\smf X),s) \ra) = \Psi_j(\la {\tau_{\la i,j \ra}}^\smf X, d \ra) = 1$.  This implies that ${\tau_{\la i,j \ra}}^\smf X \notin \MP A$, again contradicting that $\Phi_i$ and $\Psi_j$ witness that $\chi_{\MP A} \restriction \MLR \leqW \LAY$.
\end{proof}

Of course the complement of the set $\MP A$ from Theorem~\ref{thm-Sigma2NotLAY} is an example of $\Pi^0_2$ set whose characteristic function restricted to $\MLR$ is not $\leqW \LAY$.

\section{Weihrauch-below $\LAY$ versus layerwise computability}\label{sec-LWCvsLAY}

It is not difficult to see that the restriction to $\MLR$ of a layerwise computable function is Weihrauch-below $\LAY$.  With a slightly more involved argument we can also show that if $\MP A$ is any layerwise semi-decidable set, then the characteristic function of $\MP A$ restricted to $\MLR$ is $\leqW \LAY$.

\begin{Theorem}\label{thm-LWSDareLAY}
Let $\MP W$ be a universal ML-test, and let $\MP A \subseteq 2^\omega$ be $\MP W$-layerwise semi-decidable.  Then $\chi_{\MP A}\restriction\MLR \leqW \LAY$.
\end{Theorem}

\begin{proof}
Let $(\MP{U}_i)_{i \in \omega}$ be a sequence of uniformly effectively open sets that witnesses the $\MP W$-layerwise semi-decidability of $\MP A$.  We show that $\chi_{\MP A}\restriction\MLR \leqW \RD \ast \RD$, which suffices by Theorem~\ref{thm-LAYisRD} and Theorem~\ref{thm-LAYstarLAYisLAY}.

Let $g \colon \MLR \imp \MLR \times \omega$ be the function $X \mapsto \la X, \rd_{\MP W}(X) \ra$, and let $f \colon \MLR \times \omega \imp 2$ be the function
\begin{align*}
f(\la X, n \ra) = 
\begin{cases}
0 & \text{if $X \notin \MP{U}_n$}\\
1 & \text{if $X \in \MP{U}_n$}.
\end{cases}
\end{align*}
For $X \in \MLR$, we have that $f(g(X)) = f(\la X, \rd_{\MP W}(X) \ra) = 1$ if and only if $X \in \MP{U}_{\rd_{\MP W}(X)}$ if and only if $X \in \MP{U}_{\rd_{\MP W}(X)} \cap (2^\omega \setminus \MP{W}_{\rd_{\MP W}(X)})$ if and only if $X \in \MP A$.  Thus $f \circ g = \chi_{\MP A}\restriction\MLR$.  Clearly $g \leqW \RD$.  It remains to show that $f \leqW \RD$.

On input $\la X, n \ra$, $\Phi(\la X, n \ra)$ outputs the bits of $X$ while searching for a stage $s$ with $X \in \MP{U}_{n,s}$.  When such an $s$ is found, let $\sigma$ be the output produced so far, search for a string $\tau$ such that $(\forall i < s)([\sigma^\smf \tau] \subseteq \MP{W}_i)$, and then output $\sigma^\smf \tau^\smf X$.  If no such $s$ is ever found, then $\Phi(\la X, n \ra) = X$.  Let $\Psi$ be the functional
\begin{align*}
\Psi(\la \la X, n \ra, m \ra) =
\begin{cases}
0 & \text{if $X \notin \MP{U}_{n,m}$}\\
1 & \text{if $X \in \MP{U}_{n,m}$}.
\end{cases}
\end{align*}

Consider $\la X, n \ra$ with $X \in \MLR$ and $n \in \omega$.  If $X \in \MP{U}_n$, then this is witnessed at some stage $s$, and $\Phi(\la X, n \ra)$ is ML-random with $\rd_{\MP W}(\Phi(\la X, n \ra)) \geq s$.  So, in this case, if $m = \rd_{\MP W}(\Phi(\la X, n \ra))$, then $\Psi(\la \la X, n \ra, m \ra) = 1$.  On the other hand, if $X \notin \MP{U}_n$, then $\Psi(\la \la X, n \ra, m \ra) = 0$ for every $m$.  Thus $\Phi$ and $\Psi$ witness that $f \leqW \RD$.
\end{proof}

\begin{Remark}\label{rmk-ClosedNotLWSD}
If $\MP A$ is effectively open, then $\MP A$ is layerwise semi-decidable and hence 
%the characteristic function of $\MP A$ restricted to $\MLR$ is $\leqW \LAY$ 
$\chi_{\MP A}\restriction\MLR \leqW \LAY$
by either Theorem~\ref{thm-Delta02LAY} or Theorem~\ref{thm-LWSDareLAY}.  Of course, there are examples of effectively open sets that are not layerwise decidable.  In fact, Hoyrup and Rojas (\cite{MR2545900}~Proposition~6) prove that a layerwise semi-decidable set $\MP A \subseteq 2^\omega$ is layerwise decidable if and only if $\lambda(\MP A)$ is computable.  Consider then a non-empty $\Pi^0_1$ class $\MP T \subseteq \MLR$.  By \cite[Theorem~3.2.35]{Nies:book}, $\lambda(\MP T)$ is not computable.  Thus the effectively open set $\MP A = 2^\omega \setminus \MP T$ is not layerwise decidable because its measure is $1-\lambda(\MP T)$, which is not computable.  $\MP T$ itself is an example of an effectively closed set that is not layerwise semi-decidable.
\end{Remark}

\begin{Corollary}\label{cor-not_characterize}
The class of layerwise computable functions is strictly smaller than the class of functions whose restrictions to $\MLR$ are Weihrauch-reducible to $\LAY$.  The class of layerwise semi-decidable sets is strictly smaller than the class of sets whose characteristic functions restricted to $\MLR$ are Weihrauch-reducible to $\LAY$.  Hence Weihrauch reducibility to $\LAY$ does not characterize layerwise computability or even layerwise semi-decidability.
\end{Corollary}

\begin{proof}
It is easy to see that if $f$ is layerwise computable, then $f \restriction \MLR \leqW \LAY$.  Theorem~\ref{thm-LWSDareLAY} states that if $\MP A$ is layerwise semi-decidable, then $\chi_{\MP A} \restriction \MLR \leqW \LAY$.  For strictness, let $\MP T$ be an effectively closed yet not layerwise semi-decidable set as discussed in Remark~\ref{rmk-ClosedNotLWSD}.  Then $\chi_{\MP T}\restriction \MLR \leqW \LAY$ by Theorem~\ref{thm-Delta02LAY} because $\MP T$ is $\Delta^0_2$, but $\chi_{\MP T}$ is not layerwise computable and $\MP T$ is not layerwise semi-decidable.
\end{proof}

\begin{Remark}
Let $\MP U$ be an optimal ML-test.  Then by Theorem~\ref{thm-RDnotLWC} and Theorem~\ref{thm-LAYisRD}, $\rd_{\MP U}$ is another example of a function defined on $\MLR$ that is $\leqW \LAY$ but is not layerwise computable.
\end{Remark}

With the help of Theorem~\ref{thm-Delta02LAY} and the following (surely known) lemma, we also see that Weihrauch reducibility to $\LAY$ does not characterize exact layerwise computability.

\begin{Lemma}\label{lem-not_characterize2helper}
If $\MP T \subseteq 2^\omega$ is an effectively closed set of positive measure, then there is an effectively open set $\MP A \subseteq 2^\omega$ such that $\MP A \cap \MP T$ is not closed.
\end{Lemma}
\begin{proof}
If $\MP A \cap \MP T$ were closed, then it would be compact, $\MP A$ would be an open cover of $\MP A \cap \MP T$, and, by compactness, $\MP A \cap \MP T$ would be covered by the union of some finite collection of cones from~$\MP A$.  Thus it suffices to construct $\MP A$ so that no finite union of cones from $\MP A$ covers $\MP A \cap \MP T$. 

Let $T \subseteq 2^{<\omega}$ be a computable tree such that $[T] = \MP T$.  Let $n_0$ be such that $2^{-n_0} \leq \lambda(\MP T)/4$.  
The goal is to enumerate into $\MP A$, for each $n \geq n_0$, the cone $[\sigma]$ for the leftmost $\sigma \in 2^n$ such that {\em (i)} $[\sigma] \cap \MP T \neq \emptyset$ and {\em (ii)} $[\sigma] \nsubseteq [\tau]$ for every shorter $\tau$ where $[\tau]$ was enumerated into  $\MP A$. Of course, as $\MP T$ shrinks, it could happen that we 
enumerated a $\sigma$ of length $n$ at some stage $s$ because we observed $[\sigma] \cap \MP T_s \neq \emptyset$; but that later it turns out that $[\sigma] \cap \MP T = \emptyset$, that is, that {\em (i)} is violated for~$\sigma$. In this case we need to enumerate a new string $\sigma'$ of length $n$ to replace~$\sigma$. Enumerating this  new $\sigma'$ risks violating condition {\em (ii)} for some longer $\tau$ that was previously enumerated (that is, we could have $[\tau]\subseteq[\sigma']$); in such cases we will also need to replace $\tau$ by a new $\tau'$ of the same length. Further chain reactions of the same type could occur, and need to be handled by us.

\smallskip

Let us make this more formal. For the purpose of this proof, for any $\sigma \in 2^{<\omega}$, let $\sigma^+$ denote the string immediately to the right of and of equal length as $\sigma$. At stage $s = \la i, t \ra$, first check whether in the construction so far a cone $[\sigma]$ for a $\sigma$ of length $n_0+i$ has been enumerated into $\MP A$. If not, let $\sigma$ be leftmost in $2^{n_0+i}$ such that $[\sigma] \nsubseteq [\tau]$ for all $\tau$ with $[\tau] \subseteq \MP A$, and enumerate $[\sigma]$ into $\MP A$. 

 %If on the other hand there is a $\sigma \in 2^{n_0+i}$ with $[\sigma] \in \MP A$, 
 
 If yes, then choose the rightmost such $\sigma$.  Now check {\em (1)} whether $\sigma \succ \tau$ for a $\tau$ with $[\tau] \subseteq \MP A$ and {\em (2)} whether there is an extension of $\sigma$ of length $s$ in $T$.  If {\em (1)} fails and {\em (2)} holds, then~$\sigma$ is still the leftmost string of length $n_0+i$ meeting conditions {\em (i)} and {\em (ii)} at this stage, so do nothing. Otherwise, enumerate $[\sigma^+]$ into~$\MP A$.

A new cone $[\sigma^+]$ is enumerated into $\MP A$ only if either $[\sigma]$ was covered by a $[\tau] \subseteq \MP A$ for a shorter~$\tau$ or if we discover that $[\sigma] \cap \MP T = \emptyset$.  Thus at every stage, for each $n \geq n_0$, $\MP A$ contains at most one cone $[\sigma]$ such that $\sigma \in 2^n$, $[\sigma] \cap \MP T \neq \emptyset$, and $[\sigma] \nsubseteq [\tau]$ for all shorter $\tau$ with $[\tau] \in \MP A$.  Thus, at each stage, $\lambda(\MP A \cap \MP T) \leq \sum_{n \geq n_0}2^{-n} = 2^{-n+1} \leq \lambda(\MP T)/2$.  From this it follows that the construction of $\MP A$ is well-defined (it is always possible to find the required $[\sigma]$'s to enumerate into~$\MP A$) because $\MP A$ never covers $\MP T$.  Similarly, we see that no finite union of cones from $\MP A$ covers~$\MP A \cap \MP T$:  for any finite set of cones $\{[\tau_i] : i < m\}$ enumerated into $\MP A$, $\MP A$ eventually enumerates a cone $[\sigma]$ such that $\sigma$ is to the right of each of the $\tau_i$ and $[\sigma] \cap \MP T \neq \emptyset$.
\end{proof}

\begin{Theorem}\label{thm-not_characterize2}
Fix a universal ML-test $\MP U$.  Then the class of exactly $\MP U$-layerwise computable functions is strictly smaller than the class of functions whose restrictions to $\MLR$ are Weihrauch-reducible to $\LAY$.  Hence Weihrauch reducibility to $\LAY$ does not characterize exact $\MP U$-layerwise computability.
\end{Theorem}

\begin{proof}
It is easy to see that the restriction to $\MLR$ of any exactly layerwise-computable function is  Weihrauch-reducible to $\RD$. Then, by appealing to the fact that $\RD \equivW \LAY$ from Theorem~\ref{thm-LAYisRD}, we see that if $f$ is exactly $\MP U$-layerwise computable, then $f \restriction \MLR \leqW \LAY$.

For strictness, we show that there is a set $\MP B \subseteq 2^\omega$ such that $\chi_{\MP B}$ is not exactly $\MP U$-layerwise computable yet satisfies $\chi_{\MP B}\restriction\MLR \leqW \LAY$.  Let $i$ be least such that $\MP{T}_i = 2^\omega \setminus \MP{U}_i \neq \emptyset$, and note that $(\forall X \in \MP{T}_i)(\rd(X) = i)$.  By Lemma~\ref{lem-not_characterize2helper}, let $\MP A$ be an effectively open set such that $\MP B = \MP A \cap \MP{T}_i$ is not closed. 

Then $\chi_{\MP B}$ is not exactly $\MP U$-layerwise computable, but $\chi_{\MP B}\restriction\MLR \leqW \LAY$: Theorem~\ref{thm-Delta02LAY} tells us that $\chi_{\MP B}\restriction\MLR \leqW \LAY$ because $\MP B$ is $\Delta^0_2$.  Suppose for a contradiction that $\Phi$ witness that $\chi_{\MP B}$ is exactly $\MP U$-layerwise computable.  Let $\MP S = \{X \in \MP{T}_i : \Phi(X)(i) \neq 0\}$.  Then $\MP S$ is closed.  Moreover, $\MP S = \MP B$.  This is because if $X \in \MP{T}_i$, then $\rd(X) = i$, so $\Phi(X)(i) \neq 0$ if and only if $\Phi(X)(i) = 1$ if and only if $X \in \MP B$.  This contradicts that $\MP B$ is not closed.
\end{proof}

\section*{Acknowledgements}
The authors would like to thank Vasco Brattka, Guido Gherardi, and Arno Pauly for helpful discussions on some of the topics of this article.

\bibliographystyle{plain}
\bibliography{HoelzlShafer_UniversalityOptimalityRandomnessDeficiency}

\end{document}